\documentclass[12pt, A4paper]{article}
\usepackage{amsmath, amssymb, amsthm, psfrag, graphicx, array, cite}
\def\SC{{\footnotesize SS-CMC} }
\def\sC{{\footnotesize S-CMC} }
\def\C{{\footnotesize CMC} }
\theoremstyle{plain}
\newtheorem{thm}{\textrm{Theorem}}

\newtheorem{prop}{\textrm{Proposition}}
\newtheorem{cor}{\textrm{Corollary}}
\theoremstyle{definition}

\theoremstyle{remark}
\newtheorem{rem}{\textrm{Remark}}

\title{Spacelike spherically symmetric CMC hypersurfaces in Schwarzschild spacetimes (I): Construction}
\author{Kuo-Wei Lee and Yng-Ing Lee}
\date{}
\begin{document}
\maketitle
\fontsize{12}{18pt plus.5pt minus.4pt}\selectfont
\section*{\centering \small Abstract}
\begin{center}
\begin{minipage}{133mm}
\small We solve spacelike spherically symmetric constant mean curvature (\SC\!)
hypersurfaces in Schwarzschild spacetimes and analyze their asymptotic behavior near the coordinate singularity $r=2M$.
Furthermore, we join \SC hypersurfaces in the Kruskal extension to obtain complete ones and
discuss the smooth properties.
\end{minipage}
\end{center}
\footnote{2010 {\it Mathematics Subject Classification:} Primary 83C15, Secondary 83C05.}

\section{Introduction}
The Schwarzschild spacetime is the simplest model of a universe containing a star.
Its metric is a solution of the vacuum Einstein equations,
and is spherically symmetric, asymptotically flat, and Ricci flat.
A more remarkable fact is that the Schwarzschild metric is the only spherically symmetric vacuum solution
of the Einstein equations.

Spacelike constant mean curvature (\sC\!\!) hypersurfaces in spacetimes
have been considered important and interesting objects
in studying the dynamics of spacetime and in general relativity.
We refer to \cite{MT} for more discussions on the importance of \sC hypersurfaces.
From the viewpoint of geometry,
a \sC hypersurface in spacetimes has extremal surface area with fixed enclosed volume \cite{BCI}.
This property is similar to that of a compact \C hypersurface in Euclidean spaces.

In this paper, we study spacelike spherically symmetric constant
mean curvature (\SC\!\!) hypersurfaces in Schwarzschild spacetimes
and Kruskal extension. We solve \SC hypersurfaces in both exterior
and interior of the Schwarzschild spacetime, and then analyze their
asymptotic behavior, especially at $r=2M$. The Kruskal extension is
an analytic extension of the Schwarzschild spacetime. When \SC
hypersurfaces are mapped to the Kruskal extension, we find relations
between \SC hypersurfaces in exterior and interior such that they
can be joined smoothly. These statements can be seen in
Theorem~\ref{thm1}--\ref{thm3}.
Furthermore, we get all complete \SC hypersurfaces in the Kruskal extension in Theorem~\ref{thmall}.

Our motivation on studying \SC hypersurfaces is on one hand that they are easier to deal with and have explicit
expressions, and on the other hand that these examples can serve as barrier functions
for the  general non-symmetric cases. We  hope that a deep understanding of these solutions can help
us to find right formulation of other general questions in the
Schwarzschild spacetime such as Dirichlet problem and etc.
After we finished the results in this paper,
we found that the problem was also studied by Brill, Cavallo, and Isenberg in \cite{BCI},
and Malec, and \'{O} Murchadha in \cite{MO,MO2}.
However, the approaches are quite different.
Our viewpoint is purely geometrical and the explicit formula derived in this paper has
the advantage on verifying foliation properties conjectured in \cite{MO}.
This part will appear in a forthcoming paper \cite{LL2}.

The authors want to thank Quo-Shin Chi, Mao-Pei Tsui, and Mu-Tao Wang for their interests and discussions.
The first author also like to express his gratitude to Robert Bartnik and Pengzi Miao for helpful suggestions
and hospitality when he visited  Monash University in 2010.
The second author is partially supported by the NSC research grant 99-2115-M-002-008 in Taiwan.
We are also grateful to Zhuo-Bin Liang for useful suggestions.

The organization of this paper is as follows. We first give a brief summary of the Schwarzschild spacetime and
the Kruskal extension in section~\ref{Kru}. A good reference for this part is Wald's book \cite{W}.
In sections~\ref{section2}--\ref{section4},
we study \SC hypersurfaces in each region and analyze their asymptotic behavior, especially at $r=2M$.
How to glue these solutions into complete and smooth \SC
hypersurfaces are discussed in section~\ref{section5}.

\section{The Kruskal extension}\label{Kru}
The Schwarzschild spacetime, denoted by $S$, has a metric
\begin{align}
\mbox{d}s^2=-\left(1-\frac{2M}r\right)\mbox{d}t^2+\frac1{\left(1-\frac{2M}r\right)}\mbox{d}r^2
+r^2\mbox{d}\theta^2+r^2\sin^2\theta\mbox{d}\phi^2. \label{scmc2eq1}
\end{align}
We often write $h(r)=1-\frac{2M}r$. The
metric (\ref{scmc2eq1}) is not defined at $r=0$ and $r=2M$, and looks singular at both places.
But in fact, the Schwarzschild spacetime is nonsingular at $r=2M$.
It is only a coordinate singularity, which is caused merely by a breakdown of the coordinates.
There is a larger spacetime including the Schwarzschild spacetime
as a proper subset and it has a smooth metric, especially for points corresponding to $r=2M$.
Such an analytic extension was obtained by Kruskal in 1960.

\begin{prop}{\rm\cite{K, W}} \label{prop1}
The Schwarzschild metric can be written as
\begin{align}
\mathrm{d}s^2=&\frac{16M^2\mathrm{e}^{-\frac{r}{2M}}}{r}(-\mathrm{d}T^2+\mathrm{d}X^2)
+r^2\mathrm{d}\theta^2+r^2\sin^2\theta\mathrm{d}\phi^2 \nonumber \\
=&\frac{16M^2\mathrm{e}^{-\frac{r}{2M}}}{r}\mathrm{d}U\mathrm{d}V
+r^2\mathrm{d}\theta^2+r^2\sin^2\theta\mathrm{d}\phi^2, \label{SchMetric2}
\end{align}
where
\begin{align}
\left\{
\begin{array}{l}
\displaystyle(r-2M)\mathrm{e}^{\frac{r}{2M}}=X^2-T^2=VU\\
\displaystyle\frac{t}{2M}=\ln\left|\frac{X+T}{X-T}\right|=\ln\left|\frac{V}{U}\right|. \label{trans}
\end{array}
\right.
\end{align}
The metric {\rm(\ref{SchMetric2})} is nonsingular at $r=2M$.
\end{prop}

A spacetime diagram for the Kruskal extension is shown in Figure \ref{KruskalSimple}.
Each point in the Kruskal plane represents a sphere.
There is one-to-one and onto correspondence from the region {\tt I} to the Schwarzschild exterior
$r>2M$, and from the region {\tt I\!I}  to the Schwarzschild interior $0<r<2M$.
The whole Kruskal extension is the union of regions {\tt I, I\!I, I'}, and {\tt I\!I'},
where regions {\tt I'} and {\tt I\!I'} are exterior and interior of another Schwarzschild spacetime, respectively.

\begin{figure}[!h]
\psfrag{A}{\tt I}
\psfrag{B}{\tt I\!I}
\psfrag{C}{\tt I\!I'}
\psfrag{D}{\tt I'}
\psfrag{E}{$X+T=0$}
\psfrag{F}{$X^2-T^2=-2M$}
\psfrag{G}{$X-T=0$}
\psfrag{X}{$X$}
\psfrag{T}{$T$}
\centering
\includegraphics[height=82mm,width=62mm]{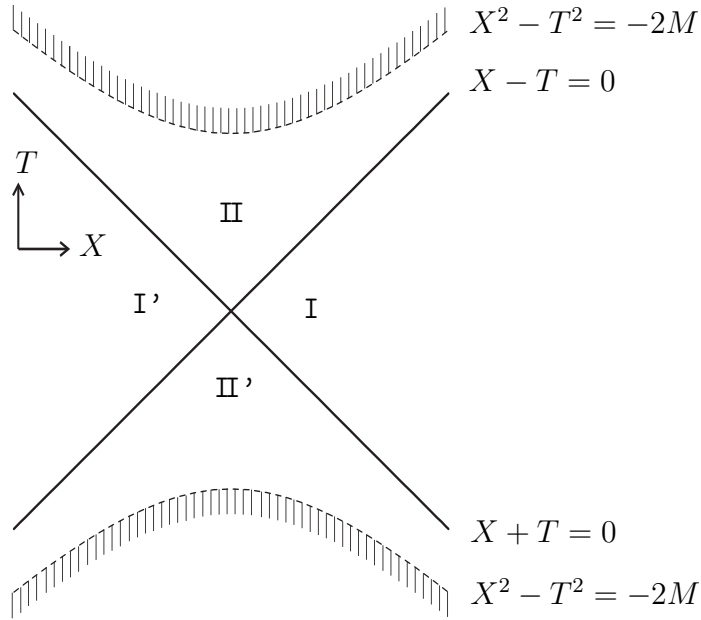}
\caption{The Kruskal extension of a Schwarzschild spacetime.}\label{KruskalSimple}
\end{figure}

From (\ref{trans}), we know that each $r=\mbox{constant}$
in the Schwarzschild spacetime is a hyperbola in the Kruskal extension,
and each $t=\mbox{constant}$ in the Schwarzschild spacetime is two half-lines starting from the origin
in the Kruskal extension.
Images of $r=\mbox{constant}$ and $t=\mbox{constant}$ under the correspondence
are illustrated in Figure~\ref{Kruskal}.

\begin{figure}[!h]
\psfrag{A}{\tt I}
\psfrag{B}{\tt I\!I}
\psfrag{C}{\tt I\!I'}
\psfrag{D}{\tt I'}
\psfrag{X}{$X$}
\psfrag{T}{$T$}
\psfrag{a}{$t=0$}
\psfrag{b}{$r=\mbox{constant}<2M$}
\psfrag{c}{$r=2M, t=\infty$}
\psfrag{d}{$r=0$}
\psfrag{e}{$t=t_0>0$}
\psfrag{f}{$r=2M, t=-\infty$}
\psfrag{g}{$r=0$}
\psfrag{h}{$r=\mbox{constant}>2M$}
\centering
\includegraphics[height=82mm,width=62mm]{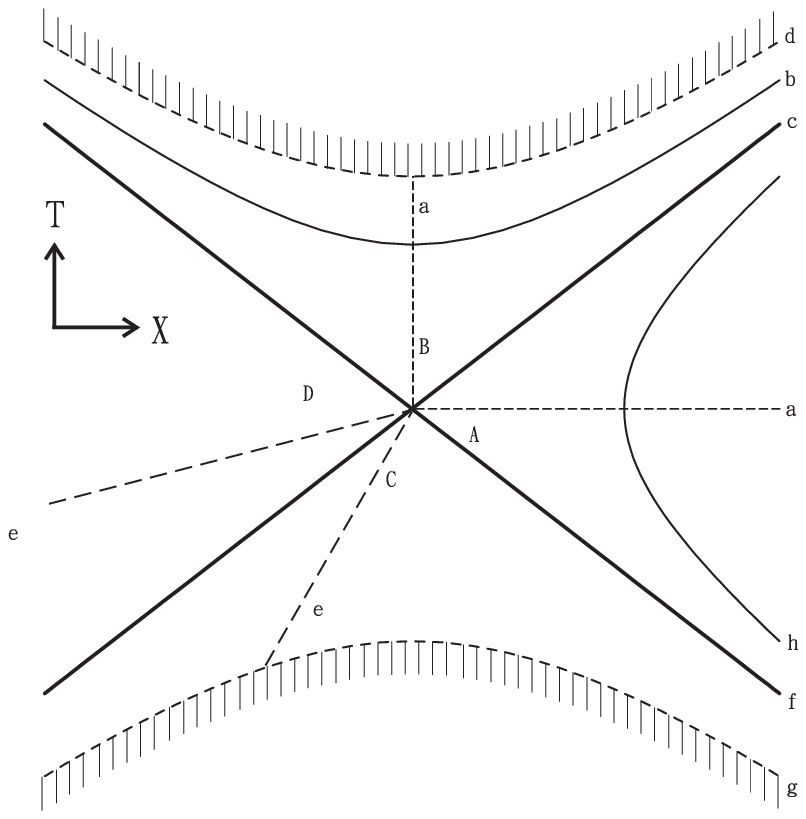}
\caption{Level sets $r=\mbox{constant}$ and $t=\mbox{constant}$.} \label{Kruskal}
\end{figure}

\begin{figure}[!h]
\psfrag{A}{\tt I}
\psfrag{B}{\tt I\!I}
\psfrag{C}{\tt I\!I'}
\psfrag{D}{\tt I'}
\psfrag{L}{$L_{+}$}
\psfrag{M}{$L_{-}$}
\psfrag{P}{$r=2M,-\infty<t<\infty$}
\psfrag{Q}{$L_{+}: r=2M,t=\infty$}
\psfrag{R}{$L_{-}: r=2M,t=-\infty$}
\psfrag{Z}{glued}
\psfrag{F}{$r=0$}
\hspace*{-3mm}
\includegraphics[height=80mm,width=130mm]{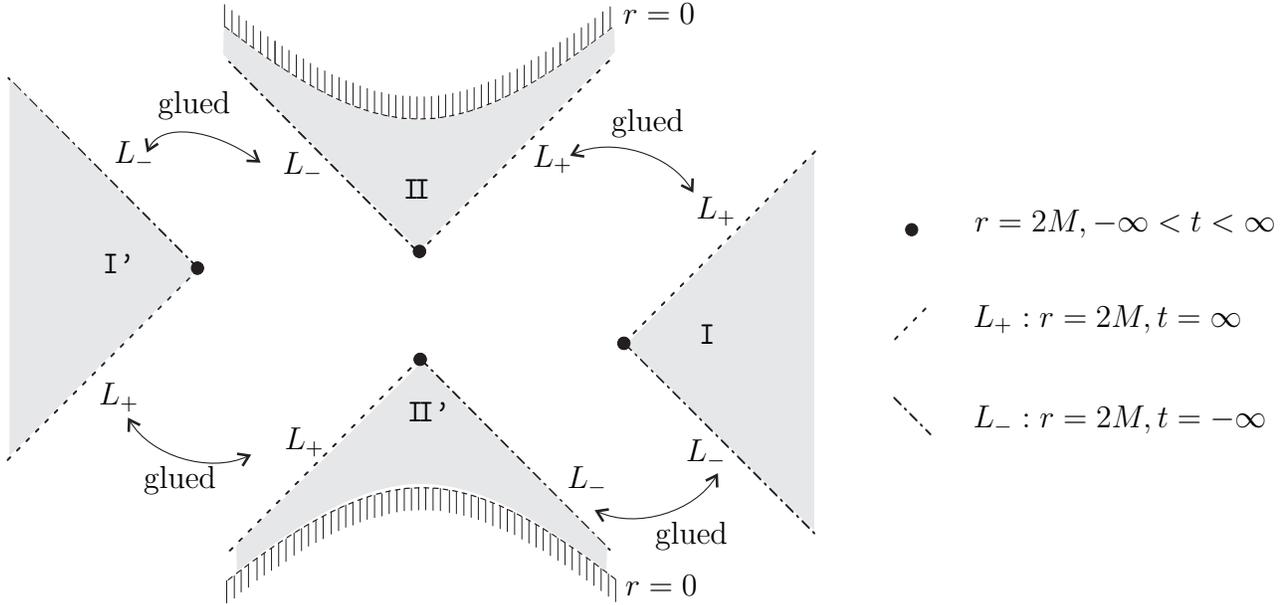}
\caption{The gluing of Schwarzschild exteriors and interiors.}
 \label{topology}
\end{figure}

\begin{figure}[h]
\psfrag{a}{$\scriptstyle u=-2$}
\psfrag{b}{$\scriptstyle u=-1$}
\psfrag{c}{$\scriptstyle u=0$}
\psfrag{d}{$\scriptstyle u=1$}
\psfrag{e}{$\scriptstyle u=2$}
\psfrag{f}{$\scriptstyle v=-2$}
\psfrag{g}{$\scriptstyle v=-1$}
\psfrag{h}{$\scriptstyle v=0$}
\psfrag{i}{$\scriptstyle v=1$}
\psfrag{j}{$\scriptstyle v=2$}
\psfrag{A}{$\scriptstyle v=-2$}
\psfrag{B}{$\scriptstyle v=-1$}
\psfrag{C}{$\scriptstyle v=0$}
\psfrag{D}{$\scriptstyle v=1$}
\psfrag{E}{$\scriptstyle v=2$}
\psfrag{F}{$\scriptstyle u=2$}
\psfrag{G}{$\scriptstyle u=1$}
\psfrag{H}{$\scriptstyle u=0$}
\psfrag{I}{$\scriptstyle u=-1$}
\psfrag{J}{$\scriptstyle u=-2$}
\psfrag{r}{$r$}
\psfrag{t}{$t$}
\psfrag{X}{$X$}
\psfrag{T}{$T$}
\hspace*{-3mm}
\includegraphics[height=82mm,width=82mm]{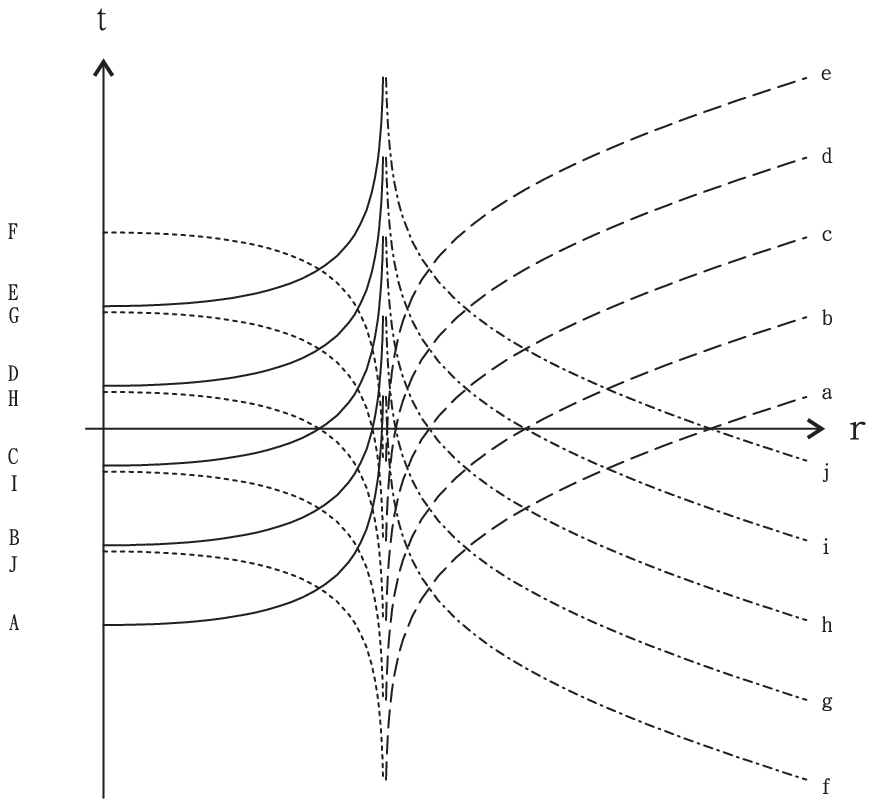}
\hspace{5mm}
\includegraphics{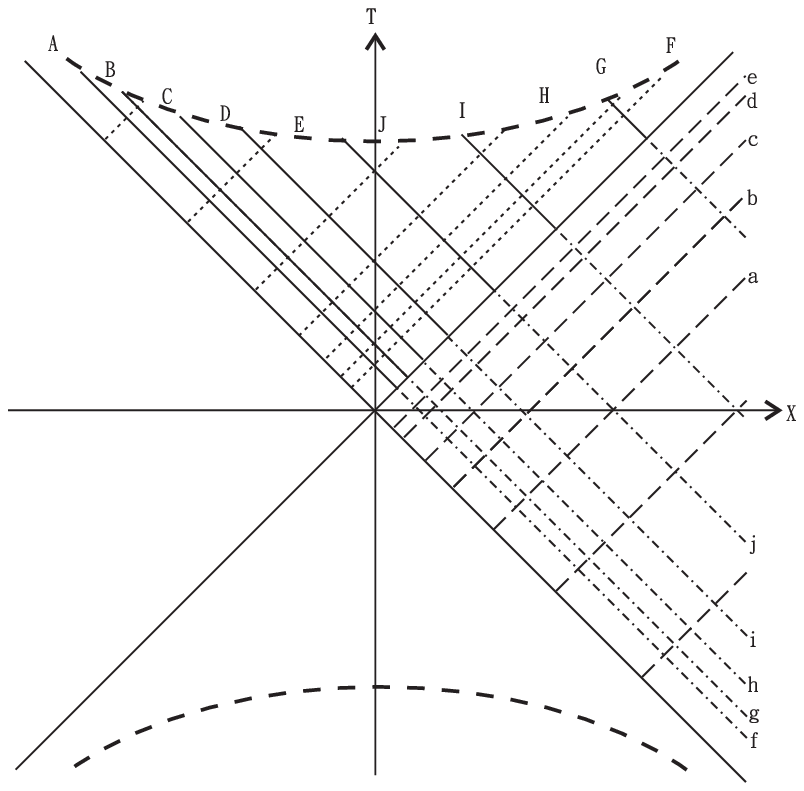}
\caption{Null geodesics in the Schwarzschild spacetime and Kruskal extension.} \label{nullgeo}
\end{figure}

Now we explain how the Schwarzschild exterior and interior change as they map into the Kruskal extension.
The boundary $r=2M,\; -\infty<t<\infty$
of the Schwarzschild exterior and interior blow down to the origin in the Kruskal extension.
On the other hand, $r=2M,\; t=\infty$ and $r=2M,\; t=-\infty$ blow up to half-lines $L_+$ and $L_{-}$
in the Kruskal extension, respectively.
The $L_+$ of {\tt I} is glued to the $L_+$ of {\tt I\!I},
and the $L_-$ of {\tt I\!I} is glued to the $L_{-}$ of {\tt I'}, and so on.
Moreover, $r=0$ is mapped to the hyperbola $X^2-T^2=-2M$ in the Kruskal extension.
This identification is pictured in Figure \ref{topology}.

The idea to the construction of the Kruskal extension is using null geodesics.
When omitting the spherically symmetric part and solving null geodesics in $t$-$r$ plane,
we can define null coordinates $u,v$ by
\begin{align*}
u=t-(r+2M\ln|r-2M|)\quad\mbox{and}\quad v=t+(r+2M\ln|r-2M|).
\end{align*}
These coordinate curves are mapped to $\pm 45^\circ$ straight lines in the Kruskal extension.
Figure~\ref{nullgeo} presents $u=\mbox{constant}$ and $v=\mbox{constant}$
in the Schwarzschild spacetime and Kruskal extension.
Furthermore, we can define null coordinates $(U,V)$ in the Kruskal extension by
\begin{align*}
\begin{array}{ccccc}
& \quad \mbox{Region {\tt I}} & \quad\mbox{Region {\tt I\!I}} & \quad \mbox{Region {\tt I'}}
& \quad \mbox{Region {\tt I\!I'}}\\
\hline
U & \quad \mbox{e}^{-\frac{u}{4M}} & \quad -\mbox{e}^{-\frac{u}{4M}}
& \quad -\mbox{e}^{-\frac{u}{4M}} & \quad \mbox{e}^{-\frac{u}{4M}} \\
V &\quad \mbox{e}^{\frac{v}{4M}}  & \quad \mbox{e}^{\frac{v}{4M}}
&\quad -\mbox{e}^{\frac{v}{4M}}  & \quad-\mbox{e}^{\frac{v}{4M}}. \\
\end{array}
\end{align*}
Direct computation from (\ref{trans}) gives the relations between $(X,T)$ and $(r,t)$ as follows:\\
\begin{tabular}{llll}
In region {\tt I}, &
$\quad\displaystyle X=\frac{\sqrt{r-2M}(\mathrm{e}^{\frac{r+t}{4M}}+\mathrm{e}^{\frac{r-t}{4M}})}{2}\quad$ &
and & $\quad\displaystyle T=\frac{\sqrt{r-2M}(\mathrm{e}^{\frac{r+t}{4M}}-\mathrm{e}^{\frac{r-t}{4M}})}{2}$.
\vspace*{5mm} \\
In region {\tt I\!I}, &
$\quad\displaystyle X=\frac{\sqrt{2M-r}(\mathrm{e}^{\frac{r+t}{4M}}-\mathrm{e}^{\frac{r-t}{4M}})}{2}\quad$ &
and & $\quad\displaystyle T=\frac{\sqrt{2M-r}(\mathrm{e}^{\frac{r+t}{4M}}+\mathrm{e}^{\frac{r-t}{4M}})}{2}$.
\vspace*{5mm} \\
In region {\tt I'}, &
$\quad\displaystyle X=-\frac{\sqrt{r-2M}(\mathrm{e}^{\frac{r+t}{4M}}+\mathrm{e}^{\frac{r-t}{4M}})}{2}\quad$ &
and & $\quad\displaystyle T=-\frac{\sqrt{r-2M}(\mathrm{e}^{\frac{r+t}{4M}}-\mathrm{e}^{\frac{r-t}{4M}})}{2}$.
\vspace*{5mm} \\
In region {\tt I\!I'}, &
$\quad\displaystyle X=-\frac{\sqrt{2M-r}(\mathrm{e}^{\frac{r+t}{4M}}-\mathrm{e}^{\frac{r-t}{4M}})}{2}\quad$ &
and & $\quad\displaystyle T=-\frac{\sqrt{2M-r}(\mathrm{e}^{\frac{r+t}{4M}}+\mathrm{e}^{\frac{r-t}{4M}})}{2}$.
\vspace*{5mm}
\end{tabular}

In this article, we always take $\partial_T$ as future directed timelike vector field.
In region {\tt I}, the vector $\partial_T$ points to the direction of increasing $t$,
while in region {\tt I\!I} it  points to the direction of decreasing $r$.
On the other hand,
$\partial_T$ points to the direction of decreasing $t$
in region {\tt I'} and points to the direction of increasing $r$ in region {\tt I\!I'}.

\section{SS-CMC solutions in region {\tt I}} \label{section2}
A vector $v$ is {\it spacelike} if $\langle v,v\rangle>0$, {\it null} if $\langle v,v\rangle=0$,
and {\it timelike} if $\langle v,v\rangle<0$.
Given a smooth function $F$ on the Schwarzschild spacetime $(S,\mathrm{d}s^2)$
with $\mathrm{d}s^2$ as in (\ref{scmc2eq1}), denote a level set of $F$ by $\Sigma=\{x\in S\,|\,F(x)=\mbox{constant}\}$,
then $\nabla F$ is a normal vector field of $\Sigma$.
If $\Sigma$ is spacelike, that is, $\Sigma$ has a positive definite metric induced from $(S,\mbox{d}s^2)$,
then $\nabla F$ forms a timelike normal vector field on $\Sigma$.
Since
\begin{align*}
\nabla F
&=g^{tt}F_t\partial_t+g^{rr}F_r\partial_r
+g^{\theta\theta}F_\theta\partial_\theta+g^{\phi\phi}F_\phi\partial_\phi \\
&=-\frac1{h(r)}F_t\partial_t+h(r)F_r\partial_r
+\frac1{r^2}F_\theta\partial_\theta+\frac1{r^2\sin^2\theta}F_\phi\partial_\phi,
\end{align*}
the spacelike condition on $\Sigma$ is equivalent to
\begin{align}
\langle\nabla F,\nabla F\rangle<0
\Leftrightarrow
-\frac1{h(r)}F_t^2+h(r)F_r^2+\frac1{r^2}F_\theta^2+\frac1{r^2\sin^2\theta}F_\phi^2<0.
\label{spacelikecond}
\end{align}

When $\Sigma$ is a level set of $F$ and is spacelike,
we can without loss of generality assume that $\nabla F$
is future directed (or replace $F$ by $-F$).
That is,
\begin{align*}
N=\frac{\nabla F}{\sqrt{-\langle\nabla F,\nabla F\rangle}}
\end{align*}
is future directed unit timelike normal vector field on $\Sigma$.

Let $\{e_i\}_{i=1}^3$ be a basis on $\Sigma$, then mean curvature of $\Sigma$ is
\begin{align*}
H=\frac13\sum_{i=1}^3g^{ij}h_{ij}
=\frac13\sum_{i=1}^3g^{ij}\langle \nabla_{e_i}N,e_j\rangle
=\frac13\sum_{i=1}^3\frac{g^{ij}}{\sqrt{-\langle\nabla F,\nabla F\rangle}}\langle \nabla_{e_i}(\nabla F),e_j\rangle.
\end{align*}

\subsection{SS-CMC solutions in region {\tt I}}
We start to study \SC solutions in the Schwarzschild exterior which maps to the region {\tt I} in the Kruskal extension.
\begin{prop} \label{scmc3prop1}
Suppose $\Sigma^{1}=(f_{1}(r),r,\theta,\phi)$ is a \SC hypersurface in the Schwarzschild exterior.
Then the mean curvature equation is
\begin{align*}
f_{1}''+\left(\left(\frac1{h}-(f_{1}')^2h\right)\left(\frac{2h}{r}+\frac{h'}2\right)
+\frac{h'}{h}\right)f_{1}'-3H\left(\frac1{h}-(f_{1}')^2h\right)^{\frac32}=0,
\end{align*}
where $h(r)=1-\frac{2M}{r}$ and $H$ is the mean curvature.
The explicit expression of $f_{1}'$ can be derived as
\begin{align*}
f_{1}'(r;H,c_{1})=\frac{l_1(r;H,c_{1})}{h(r)}\sqrt{\frac1{1+l_1^2(r;H,c_{1})}},\quad
\mbox{where}\quad l_1(r;H,c_{1})=\frac{1}{\sqrt{h(r)}}\left(Hr+\frac{c_1}{r^2}\right)
\end{align*}
for some constant $c_1$, and the integration gives
\begin{align}
f_{1}(r;H,c_{1},\bar{c}_1)
=\int_{r_1}^r\frac{l_1(r;H,c_{1})}{h(r)}\sqrt{\frac1{1+l_1^2(r;H,c_{1})}}\,\mathrm{d}r+\bar{c}_1,
\label{solnregionI}
\end{align}
where $\bar{c}_1$ is a constant and $r_1\in(2M,\infty)$ is fixed.
\end{prop}

\begin{proof}
Take $F(t,r,\theta,\phi)=-t+f_{1}(r)$ and  $\Sigma^1$ becomes a level set of $F$.
In addition, $\nabla F=\frac1{h(r)}\partial_t+f_{1}'(r)h(r)\partial_r$
is future directed because it points to the direction of increasing $t$.
The spacelike condition (\ref{spacelikecond}) is equivalent to
\begin{align}
-\frac1{h(r)}+(f_{1}'(r))^2h(r)<0 \Leftrightarrow |f_{1}'(r)h(r)|<1. \label{spacelikeExterior}
\end{align}
Thus the future directed unit timelike normal vector can be expressed as
\begin{align}
e_4=\frac{\left(\frac1{h(r)},h(r)f_{1}'(r),0,0\right)}{\sqrt{\frac1{h(r)}-(f_{1}'(r))^2h(r)}}. \label{normalext}
\end{align}
There is a canonical orthonormal frame on $\Sigma^1$
\begin{align}
e_1=\frac{(0,0,1,0)}{r},\quad
e_2=\frac{(0,0,0,1)}{r\sin\theta},\quad\mbox{and}\quad
e_3=\frac{(f_{1}'(r),1,0,0)}{\sqrt{\frac1{h(r)}-(f_{1}'(r))^2h(r)}}. \label{spacelikeext}
\end{align}
The second fundamental form of $\Sigma^1$ can be calculated directly, and we have
\begin{align*}
&h_{11}=\frac1{\left(\frac1h-(f_{1}')^2h\right)^{\frac12}}\frac{hf_{1}'}{r},\quad
h_{22}=\frac1{\left(\frac1h-(f_{1}')^2h\right)^{\frac12}}\frac{hf_{1}'}{r},\\
&h_{33}=\frac1{\left(\frac1h-(f_{1}')^2h\right)^{\frac12}}
\left(\frac{1}{\frac1h-(f_{1}')^2h}\left(f_{1}''+\frac{h'f_{1}'}{h}\right)+\frac{h'f_{1}'}2\right),
\end{align*}
and $h_{ij}=0$ for $i\neq j$.
Hence the mean curvature equation becomes
\begin{align}
f_{1}''+\left(\left(\frac1{h}-(f_{1}')^2h\right)\left(\frac{2h}{r}+\frac{h'}2\right)
+\frac{h'}{h}\right)f_{1}'-3H\left(\frac1{h}-(f_{1}')^2h\right)^{\frac32}=0,
\label{cmceq1}
\end{align}
which is a second order ordinary differential equation.

To solve $f_{1}(r)$, we define $\sin(\eta(r))=f_{1}'(r)h(r)$.
The spacelike condition (\ref{spacelikeExterior}) implies that the change of variable is meaningful,
and we can choose the range of $\eta$ in $\left(-\frac\pi2,\frac\pi2\right)$.
Equation (\ref{cmceq1}) becomes
\begin{align*}
(\tan\eta)'+\left(\frac{2}r+\frac{h'}{2h}\right)\tan\eta-3H\left(\frac{1}{h^{\frac12}}\right)=0
\Rightarrow \tan\eta=\frac1{\sqrt{h(r)}}\left(Hr+\frac{c_{1}}{r^2}\right),
\end{align*}
where $c_{1}$ is a constant.
We write $l_{1}(r;H,c_{1})=\frac1{\sqrt{h(r)}}\left(Hr+\frac{c_{1}}{r^2}\right)=\tan\eta$ for convenience.
On the other hand, since $\sin\eta=f_{1}'h$, it gives $\tan\eta=\frac{f_{1}'h}{\sqrt{1-(f_{1}'h)^2}}$.
Therefore,
\begin{align*}
\frac{f_{1}'h}{\sqrt{1-(f_{1}'h)^2}}=&l_{1}
\Rightarrow f_{1}'=\frac{l_{1}}h\sqrt{\frac{1}{1+l_{1}^2}}\quad\mbox{and} \\
f_{1}(r;H,c_{1},\bar{c}_1)
=&\int_{r_1}^r\frac{l_1(r;H,c_{1})}{h(r)}\sqrt{\frac1{1+l_1^2(r;H,c_{1})}}\mathrm{d}r+\bar{c}_1,
\end{align*}
where $\bar{c}_1$ is a constant and $r_1\in(2M,\infty)$ is a fixed number.
\end{proof}
Here are some remarks on the \SC solutions in (\ref{solnregionI}).

\begin{rem}
We can choose $r_1$ satisfying $r_1+2M\ln|r_1-2M|=0$.
\end{rem}

\begin{rem}
The sign of $l_{1}(r)$ is the same as the sign of $f_{1}'(r)$,
and the condition for $l_{1}(r)\gtreqqless 0$ is equivalent to
$Hr^3+c_{1} \gtreqqless 0$. So $f_{1}'(r)$ changes sign at most once.
More explicitly, we have
\begin{itemize}
\item[(a)] If $H>0$ and $c_{1}\geq-8M^3H$, then $f_{1}(r)$ is increasing on $r>2M$.
\item[(b)] If $H>0$ and $c_{1}<-8M^3H$, then $f_{1}(r)$ is decreasing on
$\left(2M,\left(\frac{-c_{1}}{H}\right)^{\frac13}\right)$,
and increasing on $\left(\left(\frac{-c_{1}}{H}\right)^{\frac13},\infty\right)$.
Function $f_{1}(r)$ has a unique minimum at $r=\left(\frac{-c_{1}}{H}\right)^{\frac13}$.
\item[(c)] If $H<0$ and $c_{1}\leq-8M^3H$, then $f_{1}(r)$ is decreasing on $r>2M$.
\item[(d)] If $H<0$ and $c_{1}>-8M^3H$, then $f_{1}(r)$ is increasing on
$\left(2M,\left(\frac{-c_{1}}{H}\right)^{\frac13}\right)$,
and decreasing on $\left(\left(\frac{-c_{1}}{H}\right)^{\frac13},\infty\right)$.
Function $f_{1}(r)$ has a unique maximum at $r=\left(\frac{-c_{1}}{H}\right)^{\frac13}$.
\end{itemize}
\end{rem}

\begin{rem}
The second fundamental form of $\Sigma^{1}$  with basis (\ref{spacelikeext}) satisfies
$$h_{11}=h_{22}=H+\frac{c_{1}}{r^3},\quad h_{33}=H-\frac{2c_{1}}{r^3}, \quad
\mbox{and}\quad h_{11},h_{22},h_{33}\to H \;\mbox{ as }\; r\to\infty.$$
In particular, if $c_{1}=0$, then $h_{11}=h_{22}=h_{33}=H$.
We call this hypersurface
{\it umbilical slice}.
\end{rem}

\begin{rem} The graphs of
$f_{1}(r;H,c_1,\bar{c}_1)$ for $\bar{c}_1\in\mathbb{R}$ gives a foliation in the Schwarzschild exterior.
\end{rem}

\subsection{Asymptotic behavior of SS-CMC solutions in region {\tt I}}
We analyze the asymptotic behavior of \SC solutions $f_1(r)$ in this subsection. Here we omit the dependency of
$f_1$  on $H,\,c_1,\,\bar{c}_1$ when there is no confusion.
\begin{prop}
For a \SC hypersurface $\Sigma^{1}=(f_{1}(r),r,\theta,\phi)$,
we have $\lim\limits_{r\to\infty}f_{1}'(r)=1$ if $H>0$;
$\lim\limits_{r\to\infty}f_{1}'(r)=-1$ if $H<0$,
and $\lim\limits_{r\to\infty}f_{1}'(r)=0$ if $H=0$.
Furthermore, $\Sigma^{1}$ is asymptotically null for $H\neq 0$ as $r\to\infty$,
and $\Sigma^{1}$ is asymptotically to some constant slice $(t=t_0,r,\theta,\phi)$
for $H=0$ as $r\to\infty$.
\end{prop}

\begin{proof}
Since
\begin{align*}
\lim_{r\to\infty}f_{1}'(r)
=\lim_{r\to\infty}\frac{l_{1}}{h(r)}\sqrt{\frac1{1+l_{1}^2}}
=\lim_{r\to\infty}\frac{Hr+\frac{c_{1}}{r^2}}
{\left(1-\frac{2M}{r}\right)\sqrt{1-\frac{2M}{r}+\left(Hr+\frac{c_{1}}{r^2}\right)^2}},
\end{align*}
the limit is $0$ if $H=0$, and is $\frac{H}{|H|}$ if $H\neq 0$.

We compute
\begin{align}
\langle\nabla F,\nabla F\rangle
=-\frac1{h(r)}+h(r)(f_{1}'(r))^2
=\frac{-1}{h(r)(1+l_{1}^2)}
=\frac{-1}{h(r)+\left(Hr+\frac{c_{1}}{r^2}\right)^2},\label{gradF}
\end{align}
and have $\lim\limits_{r\to\infty}\langle\nabla F,\nabla F\rangle=0$ if $H\neq0$.
\end{proof}

\begin{prop}\label{asym}
For a \SC hypersurface
$\Sigma^{1}=(f_{1}(r;H,c_{1},\bar{c}_1),r,\theta,\phi)$
in the Schwarzschild exterior, the following conclusions hold:
\begin{itemize}
\item[\rm(a)] If $c_{1}<-8M^3H$, then $f_{1}'(r)<0$ near $r=2M$,
and $f_{1}'(r)$ is of order $O((r-2M)^{-1})$.
It implies that $\lim\limits_{r\to2M^+}f_{1}(r)=\infty$.
\item[\rm(b)] If $c_{1}=-8M^3H$, then $H\cdot f_{1}'(r)\geq 0$,
and $f_{1}'(r)$ is of order $O((r-2M)^{-\frac12})$ when $H\neq 0$.
It implies that $\lim\limits_{r\to2M^+}f_{1}(r)$ is finite.
\item[\rm(c)] If $c_{1}>-8M^3H$, then $f_{1}'(r)>0$ near $r=2M$,
and $f_{1}'(r)$ is of order $O((r-2M)^{-1})$.
It implies that $\lim\limits_{r\to2M^+}f_{1}(r)=-\infty$.
\end{itemize}
When $c_{1}\neq-8M^3H$, the curve $(f_{1}(r),r)$ in $(t,r)$ spacetime is bounded by two null geodesics near $r=2M$.
For all $c_1\in\mathbb{R}$, the spacelike condition is preserved as $r\to 2M^+$.
\end{prop}

\begin{proof}
From (\ref{gradF}), we know that if $c_{1}\neq-8M^3H$, then $\lim\limits_{r\to 2M^+}\langle\nabla F,\nabla F\rangle
=\frac{-1}{\left(2MH+\frac{c_{1}}{4M^2}\right)^2}<0$, and if $c_{1}=-8M^3H$, then
\begin{align*}
\lim_{r\to 2M^+}\langle\nabla F,\nabla F\rangle
=\lim_{r\to 2M^+}\frac{-1}{\frac{r-2M}{r}+\left(\frac{H(r-2M)(r^2+2Mr+4M^2)}{r^2}\right)^2}=-\infty.
\end{align*}
Hence the spacelike condition is preserved as $r\to 2M^+$ for all $c_1\in\mathbb{R}$.

Now we  prove the asymptotic behavior of $f_1(r)$.

\noindent
\underline{If $c_{1}<-8M^3H$}, then $f_{1}'(r)<0$
(and thus $l_{1}(r)<0$) on $r\in(2M,2M+\delta_1)$ for some $\delta_1>0$.
Therefore, on $(2M,2M+\delta_1)$ by the Taylor's theorem, we have
\begin{align*}
f_{1}'&=\frac{l_{1}}{h}\sqrt{\frac{1}{1+l_{1}^2}}=\frac{1}{-h}\sqrt{1-\frac1{1+l_{1}^2}} \\
&\approx\frac{1}{-h}\left(
1-\frac12\left(\frac1{1+l_{1}^2}\right)-\frac18\left(\frac1{1+l_{1}^2}\right)^2
-\frac3{16}\left(\frac1{1+l_{1}^2}\right)^3-\cdots\right) \\
&\approx\frac1{-h}+\frac12\frac{1}{\left(h+\left(Hr+\frac{c_{1}}{r^2}\right)^2\right)}
+\frac18\frac{h}{\left(h+\left(Hr+\frac{c_{1}}{r^2}\right)^2\right)^2}+\cdots \\
&=\frac1{-h}+\mbox{remainder terms.}
\end{align*}
Remainder terms can be bounded above by $\frac{1}{2\left(Hr+\frac{c_{1}}{r^2}\right)^2}$,
so $f_1'(r)$ is of order $O((r-2M)^{-1})$. Furthermore, we have
\begin{align}
\frac1{-h(r)}\leq f'(r)\leq\frac1{-h(r)}+\frac{1}{2\left(Hr+\frac{c_{1}}{r^2}\right)^2} \label{typeIineq}
\end{align}
on $(2M,2M+\delta_1)$.
We integrate inequalities (\ref{typeIineq}) and get
\begin{align*}
\int_r^{r_1}-\frac{x}{x-2M}\mbox{d}x\leq\int_r^{r_1}&
f_{1}'(x)\mbox{d}x\leq
\int_r^{r_1}\left(-\frac{x}{x-2M}+\frac{1}{2\left(Hx+\frac{c_{1}}{x^2}\right)^2}\right)\mathrm{d}x.
\end{align*}
The integral $\int_r^{r_1}\frac{1}{2\left(Hx+\frac{c_{1}}{x^2}\right)^2}\mathrm{d}x$ is finite,
and we denote it by $C_1$.
It follows that
\begin{align*}
&\;-(r_1+2M\ln(r_1-2M))+(r+2M\ln(r-2M)) \\
\leq&\;f_1(r_1)-f_1(r) \\
\leq&\;-(r_1+2M\ln(r_1-2M))+(r+2M\ln(r-2M))+C_1 \\
\Rightarrow \quad -&(r+2M\ln(r-2M))+C_2-C_1\leq f_1(r)\leq-(r+2M\ln(r-2M))+C_2,
\end{align*}
where $C_2=f_1(r_1)+(r_1+2M\ln(r_1-2M))$.
Hence the curve $t=f_1(r)$ is bounded by two null geodesics
$t+(r+2M\ln(r-2M))=C_2-C_1$ and $t+(r+2M\ln(r-2M))=C_2$ near $r=2M$.

\noindent\underline{If $c_{1}=-8M^3H$}, then
\begin{align*}
l_{1}=\left(\frac{r}{r-2M}\right)^{\frac12}\left(\frac{Hr^3-8M^3H}{r^2}\right)
=H\left(\frac{r-2M}{r}\right)^{\frac12}\left(\frac{r^2+2Mr+4M^2}{r}\right).
\end{align*}
Direct computation gives
\begin{align*}
f_{1}'=H\left(\frac{r}{r-2M}\right)^{\frac12}
\left(\frac{r(r^2+2Mr+4M^2)^2}{r^3+H^2(r-2M)(r^2+2Mr+4M^2)^2}\right)^{\frac12},
\end{align*}
and thus $f_{1}'$ is of order $O((r-2M)^{-\frac12})$ when $H\neq 0$.

\noindent\underline{If $c_{1}>-8M^3H$},
then both $f_{1}'(r)$ and $l_{1}(r)$ are positive on $(2M,2M+\delta_2)$ for some $\delta_2>0$.
By the Taylor's theorem, we have
\begin{align*}
f_{1}'&=\frac1{h}\sqrt{1-\frac1{1+l_{1}^2}}
\approx\frac1{h}-\frac12\frac{1}{\left(h+\left(Hr+\frac{c_{1}}{r^2}\right)^2\right)}
-\frac18\frac{h}{\left(h+\left(Hr+\frac{c_{1}}{r^2}\right)^2\right)^2}-\cdots.
\end{align*}
The remainder terms are greater than $\frac{-1}{2\left(Hr+\frac{c_{1}}{r^2}\right)^2}$ on $(2M,2M+\delta_2)$.
This implies
\begin{align*}
\frac{1}{h(r)}-\frac{1}{2\left(Hr+\frac{c_{1}}{r^2}\right)^2}\leq f_1'(r)\leq\frac1{h(r)}.
\end{align*}
We integrate the above inequalities and get
\begin{align*}
\int_{r}^{r_1}\left(\frac{1}{h(x)}-\frac{1}{2\left(Hx+\frac{c_{1}}{x^2}\right)^2}\right)\mbox{d}x
\leq&\int_r^{r_1}f_1'(x)\mbox{d}x\leq\int_r^{r_1}\frac1{h(x)}\mbox{d}x.
\end{align*}
The integral $\int_{r}^{r_1}\frac{1}{2\left(Hx+\frac{c_{1}}{x^2}\right)^2}\mathrm{d}x$ is finite,
and we denote it by $C_3$.
It follows that
\begin{align*}
&\,r_1+2M\ln(r_1-2M)-(r+2M\ln(r-2M))-C_3 \\
\leq&\,f_1(r_1)-f_1(r)\\
\leq&\,r_1+2M\ln(r_1-2M)-(r+2M\ln(r-2M)) \\
\Rightarrow \quad (r&+2M\ln(r-2M))+C_4\leq f_1(r)\leq (r+2M\ln(r-2M))+C_3+C_4,
\end{align*}
where $C_4=f_1(r_1)-(r_1+2M\ln(r_1-2M))$.
Hence the curve $t=f_1(r)$ is bounded by two null geodesics
$t-(r+2M\ln(r-2M))=C_4$ and $t-(r+2M\ln(r-2M))=C_3+C_4$ near $r=2M$.
\end{proof}

Figure \ref{CMCexterior} pictures \SC hypersurfaces in the Schwarzschild exterior and their images
 in region {\tt I} of the Kruskal extension.

\begin{figure}[h]
\centering
\psfrag{A}{$r=0$}
\psfrag{C}{$r=2M$}
\psfrag{E}{$c_1<-8M^3H$}
\psfrag{F}{$c_1=-8M^3H$}
\psfrag{J}{$c_1>-8M^3H$}
\psfrag{r}{$r$}
\psfrag{t}{$t$}
\hspace*{-25mm}
\includegraphics[height=70mm,width=54mm]{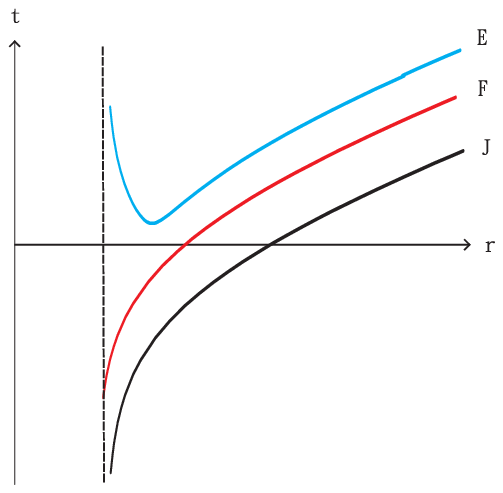}
\hspace*{25mm}
\includegraphics[height=70mm,width=54mm]{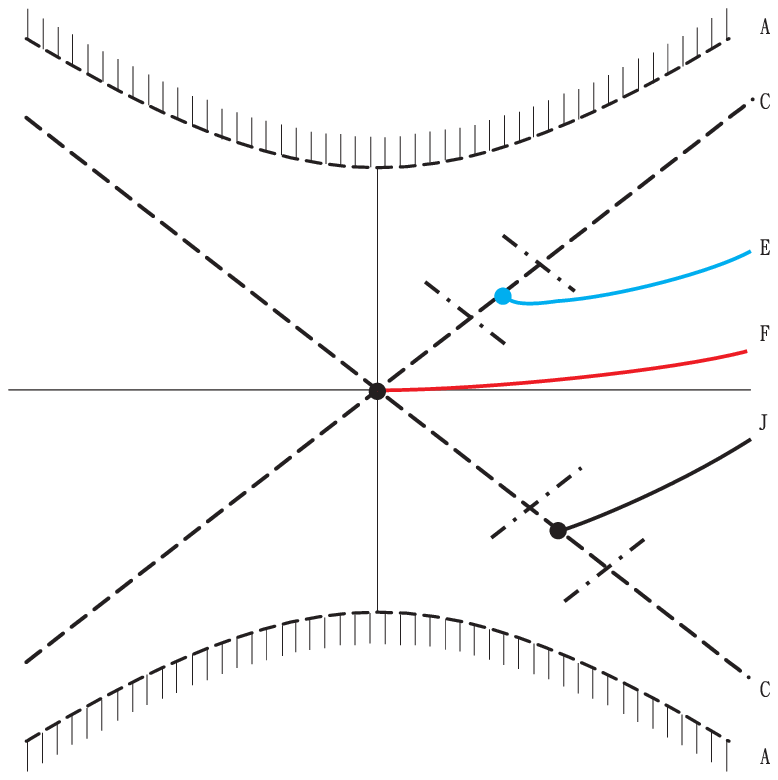}
\caption{\SC hypersurfaces in Schwarzschild exterior and region {\tt I}.} \label{CMCexterior}
\end{figure}

\section{SS-CMC solutions in region {\tt I\!I}}
We consider \SC hypersurfaces in the Schwarzschild interior in this section.
\subsection{Cylindrical hypersurfaces $r=\mbox{constant}$}
Notice that $h(r)=1-\frac{2M}r<0$ on $0<r<2M$,
so in this region $r$-direction is timelike and $t$-direction is spacelike.
Furthermore, $-\partial_r$ is future directed.
We can assume that a \SC hypersurface is written as
$(t,g(t),\theta,\phi)$ for some function $r=g(t)$.

\begin{prop}{\rm\cite{MO}}\label{cylindrical}
Each constant slice $r=r_0, r_0\in(0,2M)$ is a \SC hypersurface with
\begin{align*}
H(r_0)=\frac{2r_0-3M}{3\sqrt{r_0^3(2M-r_0)}}.
\end{align*}
These hypersurfaces are called cylindrical hypersurfaces.
\end{prop}

Cylindrical hypersurfaces are  known in \cite{MO}. Here
 we give a simple proof for completeness.

\begin{proof}
Choose $e_4=(0,-\sqrt{-h(r)},0,0)$ to be a future directed unit timelike normal vector,
and there is a canonical orthonormal frame
\begin{align*}
e_1=\frac{(0,0,1,0)}{r},\quad e_2=\frac{(0,0,0,1)}{r\sin\theta},\quad e_3=\frac{(1,0,0,0)}{\sqrt{-h(r)}}
\end{align*}
on constant slices. Since
$\nabla_{\partial_t}\partial_r=\frac{h'(r)}{2h(r)}\partial_t,
\nabla_{V}\partial_r=\frac{V}{r}$ for $V\in T_{(p,q)}\{p\}\times\mathbb{S}^2$,
we have
\begin{align*}
h_{11}
&=\langle\nabla_{e_1}e_4,e_1\rangle
=-\sqrt{-h(r)}\langle{\nabla_{e_1}\partial_r,e_1}\rangle
=-\frac{\sqrt{-h(r)}}{r}, \\
h_{22}
&=\langle\nabla_{e_2}e_4,e_2\rangle
=-\sqrt{-h(r)}\langle{\nabla_{e_2}\partial_r,e_2}\rangle
=-\frac{\sqrt{-h(r)}}{r}, \\
h_{33}
&=\langle\nabla_{e_3}e_4,e_3\rangle
=\frac{-1}{\sqrt{-h(r)}}\langle\nabla_{\partial_t}\partial_r,\partial_t\rangle
=\frac{h'(r)}{2\sqrt{-h(r)}},
\end{align*}
and $h_{ij}=0$ for $i\neq j$.
Hence the mean curvature is
\begin{align*}
H&=\frac13\left(\frac{-2\sqrt{-h(r)}}{r}+\frac{h'(r)}{2\sqrt{-h(r)}}\right)
=\frac1{3\sqrt{-h(r)}}\left(\frac{2h(r)}{r}+\frac{h'(r)}{2}\right)=\frac{2r-3M}{3\sqrt{r^3(2M-r)}},
\end{align*}
which is a constant for each fixed $r\in(0,2M)$.
\end{proof}

The following corollary is an easy consequence of Proposition \ref{cylindrical}.
\begin{cor}
Cylindrical hypersurfaces $r=r_0, r_0\in(0,2M)$ have the following properties.
\begin{itemize}
\item[\rm(a)] If $r_0\in\left(0,\frac32M\right)$, then $H(r_0)<0$ and $\lim\limits_{r\to 0^+}H(r)=-\infty$.
\item[\rm(b)] If $r_0\in\left(\frac32M,2M\right)$, then $H(r_0)>0$ and $\lim\limits_{r\to 2M^-}H(r)=\infty$.
\item[\rm(c)] If $r_0=\frac32M$, then the cylindrical hypersurface is a maximal hypersurface.
\end{itemize}
\end{cor}

\subsection{Noncylindrical SS-CMC hypersurfaces}\label{section332}
For $r=g(t)\neq\mbox{constant}$, we consider its inverse function, and denote $t=f_{2}(r)$ whenever it is defined.
Since $f_2(r)$ is obtained from the inverse function, we have $f_2'(r)\neq 0$ and will allow $f_2'(r)=\infty$ or $-\infty$.
\begin{prop}\label{prop6}
Suppose $\Sigma^{2}=(f_{2}(r),r,\theta,\phi)$ is a \SC hypersurface in Schwarzschild interior.
Then  $f_{2}'$ can be derived as
\begin{align*}
&f_{2}'=\left\{
\begin{array}{ll}
\displaystyle{\frac{1}{-h}\sqrt{\frac{l_{2}^2}{l_{2}^2-1}}} & \mbox{if } f_{2}'(r)>0 \\
\displaystyle{\frac{1}{h}\sqrt{\frac{l_{2}^2}{l_{2}^2-1}}} & \mbox{if } f_{2}'(r)<0, \\
\end{array}\right.
\quad\mbox{where}\quad l_{2}(r;H,c_2)=\frac1{\sqrt{-h(r)}}\left(-Hr-\frac{c_{2}}{r^2}\right).
\end{align*}
The function $l_2$ should satisfy $l_2>1$,
which implies $c_{2}<0$ when $H>0$ and $c_{2}<-8M^3H$ when $H<0$.
The integration of $f_2'$ gives
\begin{align}
&f_{2}^*(r;H,c_{2},\bar{c}_2)
=\int_{r_2}^r\frac{1}{-h(r)}\sqrt{\frac{l^2_{2}(r;H,c_{2})}{l_{2}^2(r;H,c_{2})-1}}\mathrm{d}r+\bar{c}_2,
\quad\mbox{or}
\label{f2positive} \\
&f_{2}^{**}(r;H,c_{2},\bar{c}_2')
=\int_{r_2'}^r\frac{1}{h(r)}\sqrt{\frac{l^2_{2}(r;H,c_{2})}{l_{2}^2(r;H,c_{2})-1}}\mathrm{d}r+\bar{c}_2'
\label{f2negative}
\end{align}
according to the sign of $f_2'(r)$, where $\bar{c}_2,\bar{c}_2'$ are constants,
and $r_2,r_2'$ are points in the domain of $f_2^*(r)$ and $f_2^{**}(r)$, respectively.
\end{prop}

\begin{rem}
In this article, when we write $f_2(r)$, it means both $f_2^*(r)$ and $f_2^{**}(r)$.
\end{rem}

\begin{proof}
First we consider the case $f_{2}'(r)>0$.
Denote $F(t,r,\theta,\phi)=-t+f_{2}(r)$, we have $\nabla F=\frac1{h(r)}\partial_t+f_{2}'(r)h(r)\partial_r$
is future directed because it is in the direction of decreasing $r$.
The spacelike condition (\ref{spacelikecond}) is equivalent to
\begin{align}
-\frac1{h(r)}+(f_{2}'(r))^2h(r)<0 \Leftrightarrow (f_{2}'(r)h(r))^2>1.
\label{spacelike2}
\end{align}
Hence future directed timelike normal vector is
\begin{align*}
e_4=\frac{\left(\frac1{h(r)},h(r)f_{2}'(r),0,0\right)}{\sqrt{\frac1{h(r)}-(f_{2}'(r))^2h(r)}},
\end{align*}
which has the same expression as (\ref{normalext}), and we can take a canonical orthonormal frame
on $\Sigma^2$ with the same expressions as (\ref{spacelikeext}).
Therefore, the mean curvature equation will be
\begin{align}
f_{2}''+\left(\left(\frac1{h}-(f_{2}')^2h\right)\left(\frac{2h}{r}+\frac{h'}2\right)
+\frac{h'}{h}\right)f_{2}'-3H\left(\frac1{h}-(f_{2}')^2h\right)^{\frac32}=0.
\label{CMCeq2}
\end{align}
To solve $f_2(r)$, from (\ref{spacelike2}) we can make change of variable by $\sec(\eta(r))=f_{2}'(r)h(r)$.
Since $h(r)=1-\frac{2M}r<0$ on $0<r<2M$, we can choose the range of $\eta$ to be $\left(\frac{\pi}2,\pi\right)$.
Then equation (\ref{CMCeq2}) becomes
\begin{align}
\left(\csc\eta\right)'+\left(\frac2r+\frac{h'}{2h}\right)\csc\eta+3H\frac1{(-h)^{\frac12}}=0
\Rightarrow
\csc\eta=\frac1{\sqrt{-h(r)}}\left(-Hr-\frac{c_2}{r^2}\right),
\label{csc}
\end{align}
where $c_2$ is a constant.
When writing $l_2(r;H,c_2)=\frac1{\sqrt{-h(r)}}\left(-Hr-\frac{c_2}{r^2}\right)=\csc\eta$,
we have
\begin{align*}
f_2'
=\frac{\sec\eta}{h}
=\frac1{-h\sqrt{1-\frac1{\csc^2\eta}}}
=\frac{1}{-h}\sqrt{\frac{l_2^2}{l_2^2-1}}.
\end{align*}
We remark that  $l_2=\csc\eta>1$ because $\eta\in\left(\frac\pi2,\pi\right)$.

For the case $f_{2}'(r)<0$, we choose $F(t,r,\theta,\phi)=t-f_{2}(r)$ such that
$\nabla F=-\frac1{h(r)}\partial_t-f_{2}'(r)h(r)\partial_r$ is future directed.
The spacelike condition is the same as (\ref{spacelike2}), but
future directed timelike normal vector is
\begin{align*}
e_4=\frac{\left(-\frac1{h(r)},-f_{2}'(r)h(r),0,0\right)}{\sqrt{\frac1{h(r)}-(f_{2}'(r))^2h(r)}}.
\end{align*}
There is a canonical orthonormal frame
\begin{align*}
e_1=\frac{(0,0,1,0)}{r},\quad e_2=\frac{(0,0,0,1)}{r\sin\theta},\quad\mbox{and}\quad
e_3=\frac{(-f_{2}'(r),-1,0,0)}{\sqrt{\frac1{h(r)}-(f_{2}'(r))^2h(r)}}
\end{align*}
on $\Sigma^2$ such that it has the same orientation as the case of $f_{2}'(r)>0$.
The second fundamental form of $\Sigma^2$ in $(S,\mbox{d}s^2)$ are
\begin{align*}
h_{11}&=-\frac1{\left(\frac1h-(f_{2}')^2h\right)^{\frac12}}\frac{hf_{2}'}{r},\quad
h_{22}=-\frac1{\left(\frac1h-(f_{2}')^2h\right)^{\frac12}}\frac{hf_{2}'}{r}, \\
h_{33}&=\frac{1}{\left(\frac1h-(f_{2}')^2h\right)^{\frac12}}
\left(\frac{-1}{\frac1h-(f_{2}')^2h}\left(f_{2}''+\frac{h'f_{2}'}{h}\right)-\frac{h'f_{2}'}2\right),
\end{align*}
and $h_{ij}=0$ for $i\neq j$.
Hence the mean curvature equation becomes
\begin{align}
f_{2}''+\left(\left(\frac1{h}-(f_{2}')^2h\right)\left(\frac{2h}{r}+\frac{h'}2\right)
+\frac{h'}{h}\right)f_{2}'+3H\left(\frac1{h}-(f_{2}')^2h\right)^{\frac32}=0.
\label{CMCeq3}
\end{align}
From (\ref{spacelike2}), we can change variable by $\sec(\eta(r))=f_{2}'(r)h(r)$, and
the range of $\eta$ can be chosen as $\left(0,\frac{\pi}2\right)$ because $h(r)=1-\frac{2M}r<0$ on $0<r<2M$.
Then (\ref{CMCeq3}) becomes
\begin{align*}
\left(\csc\eta\right)'+\left(\frac2r+\frac{h'}{2h}\right)\csc\eta+3H\frac1{(-h)^{\frac12}}=0
\Rightarrow
\csc\eta=\frac1{\sqrt{-h(r)}}\left(-Hr-\frac{c_2}{r^2}\right),
\end{align*}
which has the same expression as (\ref{csc}).
Set $l_2(r;H,c_2)=\frac1{\sqrt{-h(r)}}\left(-Hr-\frac{c_2}{r^2}\right)=\csc\eta$, then we have
\begin{align*}
f_{2}'(r;H,c_{2})=\frac{1}{h(r)}\sqrt{\frac{l_2^2(r;H,c_{2})}{l_2^2(r;H,c_{2})-1}}.
\end{align*}
We remark that $l_2=\csc\eta>1$ because $\eta\in(0,\frac\pi2)$.
\end{proof}
\subsection{Domain of SS-CMC solutions in region {\tt I\!I}}
The condition $l_2(r)>1$ will put restrictions on the domain of $f_{2}(r)$.
We have
\begin{align*}
l_2(r)=\frac1{\sqrt{-h(r)}}\left(-Hr-\frac{c_{2}}{r^2}\right)>1
\Rightarrow -Hr^3-r^{\frac32}(2M-r)^{\frac12}>c_2.
\end{align*}
Define a function $k_H(r)$ on $(0,2M)$ by
\begin{align}
k_{H}(r)=-Hr^3-r^{\frac32}(2M-r)^{\frac12}, \label{k2function}
\end{align}
then the domain of $f_2(r)$ will be
\begin{align*}
\{r\in(0,2M)|k_{H}(r)>c_{2}\}\cup\{r\in(0,2M)|k_{H}(r)=c_{2}\mbox{ and } f_{2}(r)\mbox{ is finite}\}.
\end{align*}

Now we analyze the function $k_{H}(r)$ to determine the set.
\begin{prop}\label{prop7}
Consider $k_{H}(r)$ as in {\rm(\ref{k2function})}, then $k_H(r)$ has a unique minimum point at $r=r_H$,
where $r_H$ is determined by $3Hr_{H}^{\frac32}(2M-r_H)^{\frac12}=2r_{H}-3M$.
\end{prop}

\begin{proof}
We differentiate $k_H(r)$ to get
\begin{align}
k_{H}'(r)=\frac{-r^{\frac12}}{(2M-r)^{\frac12}}\left(3Hr^{\frac32}(2M-r)^{\frac12}+3M-2r\right).
\label{diffkH}
\end{align}
Denote $\bar{k}_H(r)=3Hr^{\frac32}(2M-r)^{\frac12}$, then
\begin{align*}
\bar{k}'_H(r)=\frac{3Hr^{\frac12}}{(2M-r)^{\frac12}}(3M-2r).
\end{align*}
It implies $\bar{k}'_H(\frac32M)=0$, and $\bar{k}_H(r)$ is monotone on $(0,\frac32M)$ and $(\frac32M,2M)$.
Furthermore, $\bar{k}_H(r)$ and the function $p(r)=2r-3M$ intersect at $r=r_H$. (See Figure \ref{KHPNnew4}.)
That is, $3Hr_{H}^{\frac32}(2M-r_H)^{\frac12}=2r_{H}-3M$ and $k_H'(r_H)=0$, so $r_H$ is the critical point of $k_H(r)$.

\begin{figure}[h]
\centering
\psfrag{A}{$\frac32M$}
\psfrag{B}{$2M$}
\psfrag{C}{$r_H$}
\psfrag{G}{ }
\psfrag{D}{ }
\psfrag{E}{$c_{2}$}
\psfrag{t}{$t$}
\psfrag{r}{$r$}
\psfrag{F}{$\bar{k}_{H}(r)$}
\psfrag{H}{$k_{H}(r)$}
\psfrag{M}{(a) $H>0$}
\psfrag{N}{(b) $H<0$}
\includegraphics[height=130mm,width=160mm]{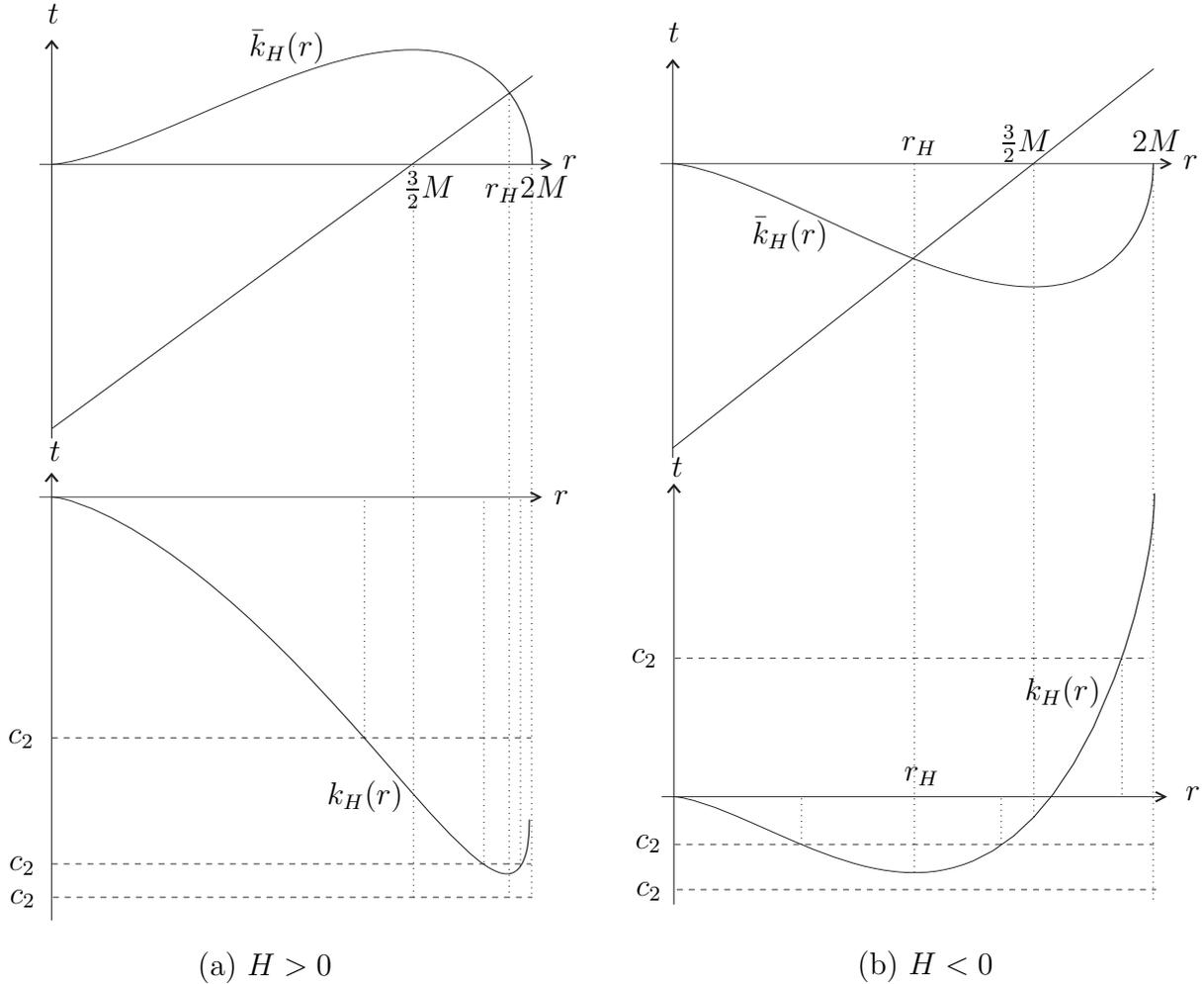}
\caption{Graphs of $\bar{k}_{H}(r)$, $p(r)=2r-3M$, $k_{H}(r)$, and horizontal lines $l(r)=c_2$.} \label{KHPNnew4}
\end{figure}
\end{proof}

\begin{prop}\label{prop8}
Denote $c_H=\min\limits_{r\in(0,2M)}k_H(r)=k_H(r_H)$,
where $k_H(r)$ is as in {\rm(\ref{k2function})},
and $r_H$ is as in Proposition \ref{prop7}.
There are three types of noncylindrical \SC hypersurfaces
$\Sigma^2=(f_2(r),r,\theta,\phi)$ according to the value of $c_2$,
where $f_2(r)=f_2^*(r;H,c_2,\bar{c}_2)$ or $f_2^{**}(r;H,c_2,\bar{c}_2')$.
\begin{itemize}
\item[\rm(a)] If $c_2<c_H$, then $f_{2}(r)$ is defined on $(0,2M)$.
\item[\rm(b)] If $c_2=c_H$, then $f_{2}(r)$ is defined on $(0,r_H)\cup(r_H,2M)$.
\item[\rm(c)] If $c_H<c_2<\max(0,-8M^3H)$,
then $f_{2}(r)$ is defined on $(0,r']$ or $[r'',2M)$ for some
$r'$ and $r''$, which depend on $H$ and $c_2$.
When we take $r_2=r_2'=r' (\mbox{or } r'')$ and $\bar{c}_2=\bar{c}_2'$
in {\rm(\ref{f2positive})} and {\rm(\ref{f2negative})},
$\Sigma^2=(f_{2}^{*}(r;H,c_2,\bar{c}_2)\cup f_2^{**}(r;H,c_2,\bar{c}_2'),r,\theta,\phi)$
is a complete \SC hypersurface in the Schwarzschild interior.
\end{itemize}
\end{prop}

\begin{proof}\mbox{}
\begin{itemize}
\item[(a)] If $c_{2}<c_H$,
then $l_2(r)>1$ for all $r\in(0,2M)$, which implies $f_2(r)$ is defined on $(0,2M)$.
\item[(b)] If $c_{2}=c_H$, then $f_{2}(r)$ is defined on $(0,r_H)\cup(r_H,2M)$.
We need to check the behavior of $f_{2}(r)$ as $r \to r_H$.
First, we know $\lim\limits_{r\to r_H}f_{2}'(r)=\infty$ or $-\infty$ because $l_2(r)=1$.
Next, noting that $c_2=-Hr_H^3-r_H^{\frac32}(2M-r_H)^{\frac12}$ and
\begin{align*}
f_{2}'(r)
=\frac{-Hr^3-c_{2}}{h(r)\sqrt{(-Hr^3-c_{2})^2+r^3(r-2M)}}\quad\mbox{or}\quad
\frac{-Hr^3-c_{2}}{-h(r)\sqrt{(-Hr^3-c_{2})^2+r^3(r-2M)}},
\end{align*}
we expand $(-Hr^3-c_{2})^2+r^3(r-2M)$ in the power of $(r-r_H)$ to attain
\begin{align*}
&\sqrt{(-Hr^3-c_{2})^2+r^3(r-2M)} \\
=&\sqrt{P_1(r;r_H)(r-r_H)^2+2r_H^2\left(-3Hr_H^{\frac32}(2M-r_H)^{\frac12}+2r_H-3M\right)(r-r_H)} \\
=&\sqrt{P_1(r;r_H)(r-r_H)^2},
\end{align*}
where $P_1(r;r_H)$ is a polynomial.
The last equality holds because $r_H$ is the critical point of $k_H(r)$.
Thus $f_{2}'(r)\sim O((r-r_H)^{-1})$ and $\lim\limits_{r\to r_H}f_{2}(r)=\infty$ or $-\infty$.
That is, the domain of $f_{2}(r)$ is $(0,r_H)\cup (r_H,2M)$.

\item[(c)] If $c_H<c_{2}<\max(0,-8M^3H)$,
then $f_{2}(r)$ is defined on $(0,r')$ or $(r'',2M)$.
Here we only discuss the case at $r'$ because the case at $r''$ is similar.
First, we know $\lim\limits_{r\to r'}f_{2}'(r)=\infty$ or $-\infty$.
Next, since $c_2=-H(r')^3-(r')^{\frac32}(2M-r')^{\frac12}$ is not a critical value of $k_H(r)$,
 the expansion of $(-Hr^3-c_{2})^2+r^3(r-2M)$ in the power of $(r-r')$ becomes
\begin{align*}
&\sqrt{(-Hr^3-c_{2})^2+r^3(r-2M)} \\
=&\sqrt{P_2(r;r')(r-r')^2
+2(r')^2\left(-3H(r')^{\frac32}(2M-r')^{\frac12}+2r'-3M\right)(r-r')},
\end{align*}
where $P_2(r;r')$ is a polynomial, and $-3H(r')^{\frac32}(2M-r')^{\frac12}+2r'-3M\neq 0$.
It implies $f_{2}'(r)\sim O((r-r')^{-\frac12})$,
and $\lim\limits_{r\to r'}f_2(r)$ is a finite value.
Domain of $f_{2}(r)$ can be extended to $r=r'$.
When taking $r_2=r_2'=r'$ and $\bar{c}_2=\bar{c}_2'$,
we have $f_{2}^{*}(r';H,c_2,\bar{c}_2)=f_2^{**}(r';H,c_2,\bar{c}_2')=\bar{c}_2$ and
$\Sigma^2=(f_{2}^{*}(r;H,c_2,\bar{c}_2)\cup f_2^{**}(r;H,c_2,\bar{c}_2'),r,\theta,\phi)$
is a complete \SC hypersurface in the Schwarzschild interior.
\end{itemize}
\end{proof}

\begin{prop}\label{proposition9}
In case {\rm(c)} of Proposition \ref{prop8}, the \SC hypersurface $\Sigma^2$ is $C^\infty$.
\end{prop}

\begin{proof}
It suffices to check the smoothness of $\Sigma^2$ at the joint point,
and here we show the case of $r_2=r_2'=r'$. The case of $r_2=r_2'=r''$ is similar.
Noting that $\bar{c}_2=\bar{c}_2'$ and $r<r'$,
we have $f^*(r)\leq\bar{c}_2$, $f^{**}(r)\geq\bar{c}_2$, and $f^*(r')=f^{**}(r')=\bar{c}_2$.
Hence when rewrite the surface as a graph of $r=g(t)$, we have $g(\bar{c}_2)=r'$ and
its inverse corresponds to  $t=f_2^*(r)$ for $t\leq\bar{c}_2$
and to $t=f_2^{**}(r)$ for $t\geq\bar{c}_2$.
Direct computation gives
\begin{align*}
g^{(2k+1)}(t)=&\left\{
\begin{array}{ll}
\sum_{i=0}^kA_{k,i}(l_2^2-1)^{i+\frac12} & \mbox{ if } t<\bar{c}_2\\
(-1)^{2k+1}\sum_{i=0}^kA_{k,i}(l_2^2-1)^{i+\frac12} & \mbox{ if } t>\bar{c}_2,
\end{array}\right. \\
g^{(2k)}(t)=&\left\{
\begin{array}{ll}
\sum_{i=0}^kB_{k,i}(l_2^2-1)^{i} & \mbox{ if } t<\bar{c}_2\\
(-1)^{2k}\sum_{i=0}^kB_{k,i}(l_2^2-1)^{i} & \mbox{ if } t>\bar{c}_2,
\end{array}\right.
\end{align*}
where $A_{k,i}$ and $B_{k,i}$ are functions of $h, l_2$ and their derivatives with respective to $r$.
As $t\to\bar{c}_2$, we have $r\to r'$ and $\lim\limits_{r\to r'}l_2^2-1=0$, it implies that
\begin{align*}
\lim_{t\to\bar{c}_2^-}g^{(2k+1)}(t)=\lim_{t\to\bar{c}_2^+}g^{(2k+1)}(t)=0\quad\mbox{and}\quad
\lim_{t\to\bar{c}_2^-}g^{(2k)}(t)=\lim_{t\to\bar{c}_2^+}g^{(2k)}(t)=B_{k,0}.
\end{align*}
Hence $\Sigma^2$ is smooth.
\end{proof}

\subsection{Asymptotic behavior of SS-CMC solutions in region {\tt I\!I}}
Next, we discuss the asymptotic behavior of \SC hypersurfaces in Schwarzschild interior that will be needed
in section~\ref{section5}.
\begin{prop}\label{asym2}
For a \SC hypersurface $\Sigma^{2}=(f_{2}(r;H,c_{2},\bar{c}_2),r,\theta,\phi)$
in Schwarzschild interior with $c_{2}<-8M^3H$, we have
$f_{2}'(r)$ is of order $O((2M-r)^{-1})$  as $r\to 2M^-$.
It implies that $\lim\limits_{r\to 2M^-}f_{2}(r)=\infty$ or $-\infty$,
and the curve $(f_{2}(r),r)$ in $(t,r)$ plane is bounded by two null geodesics as $r\to 2M^-$.
Furthermore, the spacelike condition is preserved as $r\to 2M^-$.
\end{prop}

\begin{proof}
Since $c_2<-8M^3H$, $f_2(r)$ is defined on $(2M-\delta_3,2M)$ for some $\delta_3>0$.
We only need to consider the case $f_2'(r)>0$ because of symmetry.
On one hand, since
\begin{align*}
f_2'(r)=\frac1{(-h)\sqrt{1-\frac{r^3(2M-r)}{(-Hr^3-c_{2})^2}}}\geq\frac1{-h},
\end{align*}
we have
\begin{align*}
&\int_{r_2}^rf_2'(x)\mbox{d}x
\geq-(x+2M\ln(2M-x))|_{x=r_2}^{x=r} \\
\Rightarrow f_2(r)\geq&f_2(r_2)+(r_2+2M\ln(2M-r_2)-(r+2M\ln(2M-r)) \\
=&-(r+2M\ln(2M-r))+C_5,
\end{align*}
where $C_5=f_2(r_2)+r_2+2M\ln(2M-r_2)$.
The curve $t=f_2(r)$ is bounded below by the null geodesic $t+(r+2M\ln(2M-r))=C_5$ near $r=2M$.

On the other hand, because $\frac1{l_{2}^2}$ is very small  near $r=2M,$
by Taylor's expansion we get
\begin{align*}
\sqrt{1-\frac1{l_{2}^2}}
&\approx1-\frac12\left(\frac1{l_{2}^2}\right)-\frac18\left(\frac1{l_{2}^2}\right)^2-\cdots
\geq1-\left(\frac1{l_{2}^2}\right)-\left(\frac1{l_{2}^2}\right)^2-\cdots \\
&=1-\frac1{l_{2}^2-1}=1-\frac{-h}{\left(-Hr-\frac{c_{2}}{r^2}\right)^2-(-h)}
\end{align*}
There is a constant $C_6>0$ such that $C_6\left(\left(-Hr-\frac{c_{2}}{r^2}\right)^2-2(-h)\right)>1$ on
$(2M-\delta_4,2M)$, a subset of $(2M-\delta_3,2M)$. That is, we have
\begin{align*}
\frac1{\left(-Hr-\frac{c_{2}}{r^2}\right)^2-(-h)}<\frac{C_6}{1+C_6(-h)}
\end{align*}
and
\begin{align*}
\sqrt{1-\frac1{l_{2}^2}}
\geq1-\frac{C_6(-h)}{1+C_6(-h)}
=\frac1{1+C_6(-h)}.
\end{align*}
Thus
\begin{align*}
f_2'(r)&=\frac1{(-h)\sqrt{1-\frac1{l_{2}^2}}}
\leq\frac1{(-h)}(1+C_6(-h))=\frac1{-h}+C_6,
\end{align*}
which integrates to
\begin{align*}
\int_{r_2}^r&f_2'(x)\mbox{d}x\leq\int_{r_2}^r\left(\frac1{-h(x)}+C_6\right)\mbox{d}x.
\end{align*}
Hence
\begin{align*}
f_2(r)&\leq f_2(r_2)-(x+2M\ln(2M-r))|_{x=r_2}^{x=r}+C_6(r-r_2) \\
&=-(r+2M\ln(2M-r))+C_5+C_7,
\end{align*}
where $C_5=f_2(r_2)+(r_2+2M\ln(2M-r_2))$ and $C_7=C_6(2M-r_2)$.
The curve $t=f_2(r)$ is bounded above by the null geodesic
$t+(r+2M\ln(2M-r))=C_5+C_7$ near $r=2M$.

Spacelike property can be extended at $r=2M$ because
\begin{align*}
\lim_{r\to 2M^-}\langle\nabla F,\nabla F\rangle
=\lim_{r\to 2M^-}\frac{-1}{\left(-Hr-\frac{c_{2}}{r^2}\right)^2+h(r)}
=\frac{-1}{\left(-2MH-\frac{c_{2}}{4M^2}\right)^2}<0.
\end{align*}
\end{proof}

Figure \ref{CMCinterior} pictures \SC hypersurfaces in Schwarzschild interior and their images in region
{\tt I\!I} of Kruskal extension.

\begin{figure}[!h]
\centering
\psfrag{A}{$r=0$}
\psfrag{B}{$r=r_H$}
\psfrag{C}{$r=2M$}
\psfrag{L}{$c_2>c_H$}
\psfrag{M}{$c_2=c_H$}
\psfrag{N}{$c_2<c_H$}
\psfrag{P}{$2M$}
\psfrag{Q}{$r_H$}
\hspace*{-10mm}
\includegraphics[height=85mm,width=143mm]{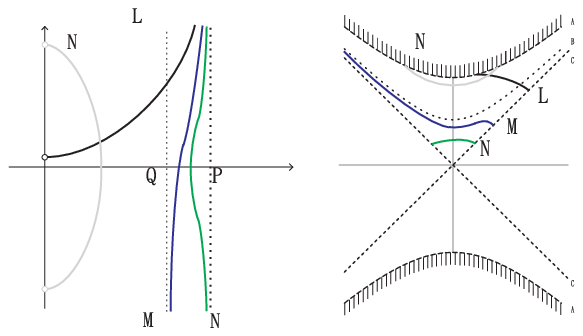}
\caption{\SC hypersurfaces in Schwarzschild interior and region {\tt I\!I}.} \label{CMCinterior}
\end{figure}

\section{SS-CMC Solutions in region {\tt I'} and {\tt I\!I'}}\label{section4}
When the Schwarzschild exterior and interior map to region {\tt I} and {\tt I\!I} in the Kruskal extension,
future directed timelike directions are directions of increasing $t$ and decreasing $r$, respectively.
However, when the Schwarzschild exterior and interior map to  region {\tt I'} and {\tt I\!I'},
future directed timelike directions are directions of decreasing $t$ and increasing $r$, respectively.
Therefore, we need to modify the discussions in section \ref{section2} and 4
according to these differences for the \SC solutions in region {\tt I'} and {\tt I\!I'}.

The constant mean curvature equation of a \SC hypersurface
$\Sigma^{3}=(f_{3}(r),r,\theta,\phi)$ which maps to region {\tt I'} of the Kruskal extension is
\begin{align}
f_{3}''
+\left(\left(\frac1{h}-(f_{3}')^2h\right)\left(\frac{2h}{r}+\frac{h'}2\right)
+\frac{h'}{h}\right)f_{3}'
+3H\left(\frac1{h}-(f_{3}')^2h\right)^{\frac32}=0.
\label{cmceq4}
\end{align}
Solutions to the equation (\ref{cmceq4}) would be
\begin{align*}
&f_{3}'(r;H,c_{3})=\frac{l_{3}(r;H,c_3)}{h(r)}\sqrt{\frac{1}{1+l_{3}^2(r;H,c_3)}},\quad\mbox{where}\quad
l_{3}(r;H,c_3)=\frac{1}{\sqrt{h(r)}}\left(-Hr-\frac{c_{3}}{r^2}\right),
\end{align*}
and
\begin{align}
f_{3}(r;H,c_{3},\bar{c}_3)
=\int_{r_3}^r\frac{l_{3}(r;H,c_{3})}{h(r)}\sqrt{\frac1{1+l_{3}^2(r;H,c_{3})}}\mathrm{d}r+\bar{c}_3,
\label{forumla3}
\end{align}
where $c_3$ and $\bar{c}_3$ are constants, and $r_3\in(2M,\infty)$ is fixed.
We remark that for given constant mean curvature
 $H$, if $c_1=c_3$, then $f_1'(r)=-f_3'(r)$.

Similarly, when we consider \SC hypersurfaces in another Schwarzschild interior that maps to
region {\tt I\!I'} of the Kruskal extension,  each constant slice $r=r_0, r_0\in(0,2M)$ are \SC solutions.

Moreover, given a \SC hypersurfaces $\Sigma^{4}=(f_{4}(r),r,\theta,\phi)$ which maps to region {\tt I\!I'},
the constant mean curvature equation of $f_{4}(r)$ is
\begin{align}\left\{
\begin{array}{ll}
\displaystyle f_{4}''
+\left(\left(\frac1h-(f_{4}')^2h\right)\left(\frac{2h}{r}+\frac{h'}2\right)
+\frac{h'}{h}\right)f_{4}'
+3H\left(\frac1h-(f_{4}')^2h\right)^{\frac32}=0 & \mbox{if } f_{4}'(r)>0  \\
\displaystyle f_{4}''
+\left(\left(\frac1h-(f_{4}')^2h\right)\left(\frac{2h}{r}+\frac{h'}2\right)
+\frac{h'}{h}\right)f_{4}'
-3H\left(\frac1h-(f_{4}')^2h\right)^{\frac32}=0 & \mbox{if } f_{4}'(r)<0.
\end{array}\right.\label{CMCeq31}
\end{align}
We can solve (\ref{CMCeq31}) to get
\begin{align*}
f_{4}'(r)=\left\{
\begin{array}{ll}
\displaystyle\frac{1}{-h}\sqrt{\frac{l_{4}^2}{l_{4}^2-1}}, & \mbox{if } f_{4}'(r)>0 \\
\displaystyle\frac{1}{h}\sqrt{\frac{l_{4}^2}{l_{4}^2-1}},  & \mbox{if } f_{4}'(r)<0,
\end{array}\right.
\quad\mbox{where}\quad
l_{4}(r;H,c_{4})=\frac1{\sqrt{-h(r)}}\left(Hr+\frac{c_{4}}{r^2}\right).
\end{align*}
The integration of $f_4'(r)$ gives
\begin{align}
&f_{4}^{*}(r;H,c_{4},\bar{c}_4)
=\int_{r_4}^r\frac{1}{-h(r)}\sqrt{\frac{l^2_{4}(r;H,c_{4})}{l_{4}^2(r;H,c_{4})-1}}\mathrm{d}r+\bar{c}_4,
\quad\mbox{or} \label{f4positive}\\
&f_{4}^{**}(r;H,c_{4},\bar{c}_4')
=\int_{r_4'}^r\frac{1}{h(r)}\sqrt{\frac{l^2_{4}(r;H,c_{4})}{l_{4}^2(r;H,c_{4})-1}}\mathrm{d}r+\bar{c}_4'
\label{f4negative}
\end{align}
according to the sign of $f_4'(r)$, where $c_{4}, \bar{c}_4, \bar{c}_4'$ are constants,
and $r_4, r_4'$ are fixed numbers in the domain of $f_4^{*}(r)$ and $f_4^{**}(r)$, respectively.
The function $l_{4}(r)$ should satisfy $l_{4}(r)>1$,
which implies $c_4>-8M^3H$ when $H>0$ and $c_4>0$ when $H<0$.
In addition, we allow $f_{4}'(r)=\pm\infty$ at some point.

In this article, when we write $f_4(r)$, it means both $f_4^{*}(r)$ and $f_4^{**}(r)$.

The condition $l_{4}(r)>1$ will put restrictions on the domain of $f_4(r)$. We have
\begin{align*}
l_4(r)=\frac1{\sqrt{-h(r)}}\left(Hr+\frac{c_4}{r^2}\right)>1 \Rightarrow -Hr^3+r^{\frac32}(2M-r)^{\frac12}<c_{4}.
\end{align*}
Define
\begin{align}
\tilde{k}_{H}(r)=-Hr^3+r^{\frac32}(2M-r)^{\frac12}.\label{ktilde}
\end{align}
Then the domain of $f_4(r)$ will be
\begin{align*}
\{r\in(0,2M)|\tilde{k}_{H}(r)<c_{4}\}
\cup\{r\in(0,2M)|\tilde{k}_{H}(r)=c_{4} \mbox{ and } f_{4}(r) \mbox{ is finite}\}.
\end{align*}
By similar arguments as in Proposition \ref{prop7},
we can analyze $\tilde{k}_{H}(r)$ and illustrate its graph according to the sign of $H$ in Figure \ref{KHPNnew3}.
Our conclusion is
\begin{prop} \label{proposition11}
Consider $\tilde{k}_{H}(r)$ as in {\rm(\ref{ktilde})}, then $\tilde{k}_H(r)$ has a unique maximum point at $r=R_H$,
where $R_H$ is determined by $-3HR_{H}^{\frac32}(2M-R_H)^{\frac12}=2R_H-3M$.
\end{prop}

\begin{figure}[h]
\centering
\psfrag{A}{$\frac32M$}
\psfrag{B}{\hspace*{1mm}$2M$}
\psfrag{C}{\hspace*{-1mm}$R_H$}
\psfrag{G}{ }
\psfrag{D}{ }
\psfrag{E}{$c_{4}$}
\psfrag{t}{$t$}
\psfrag{r}{$r$}
\psfrag{H}{$\tilde{k}_{H}(r)$}
\psfrag{J}{(a) $H>0$}
\psfrag{K}{(b) $H<0$}
\includegraphics[height=75mm,width=160mm]{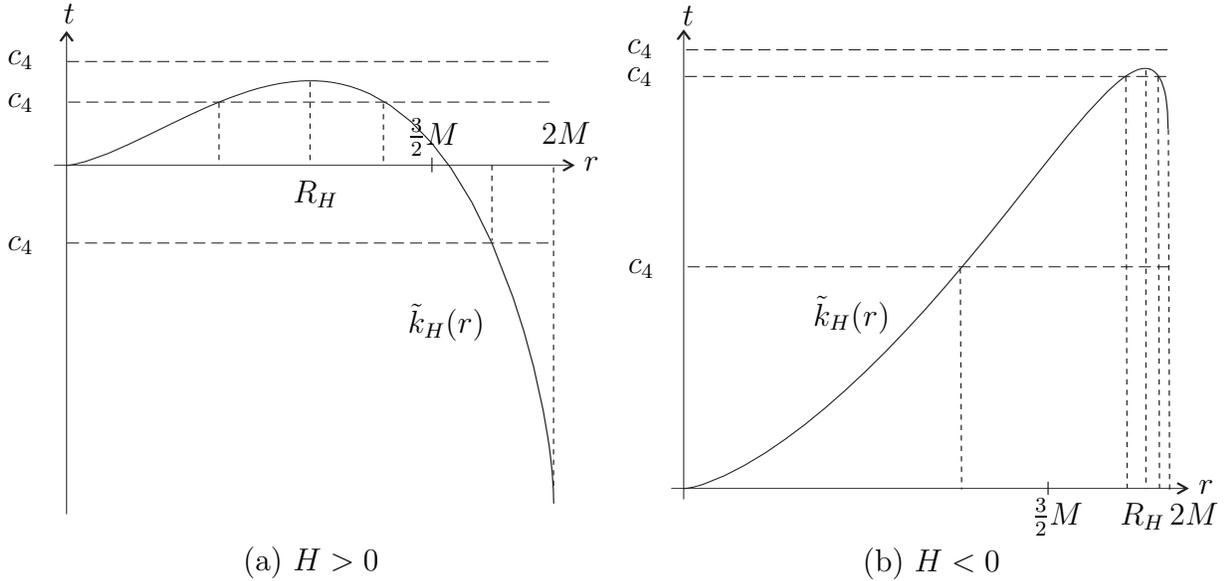}
\caption{Graphs of $\tilde{k}_{H}(r)$ and horizontal lines $l(r)=c_4$.} \label{KHPNnew3}
\end{figure}

Corresponding to Proposition \ref{prop8} and Proposition \ref{proposition9},
the following results can be proved by the same method.

\begin{prop}\label{prop11}
Denote $C_H=\max\limits_{r\in(0,2M)}\tilde{k}_H(r)=\tilde{k}_H(R_H)$,
where $\tilde{k}_H(r)$ is as in {\rm(\ref{ktilde})},
and $R_H$ is as in Proposition \ref{proposition11}.
There are three types of noncylindrical \SC hypersurfaces
$\Sigma^4=(f_4(r),r,\theta,\phi)$ according to the value of $c_4$,
where $f_4(r)=f_4^{*}(r;H,c_4,\bar{c}_4),r,\theta,\phi)$ or $(f_4^{**}(r;H,c_4,\bar{c}_4'),r,\theta,\phi)$.
\begin{itemize}
\item[\rm(a)] If $c_4>C_H$, then $f_{4}(r)$ is defined on $(0,2M)$.
\item[\rm(b)] If $c_4=C_H$, then $f_{4}(r)$ is defined on $(0,R_H)\cup(R_H,2M)$.
\item[\rm(c)] If $\min(0,-8M^3H)<c_4<C_H$,
then $f_{4}(r)$ is defined on $(0,r']$ or $[r'',2M)$ for some
$r'$ and $r''$, which depend on $H$ and $c_4$.
When we take $r_4=r_4'=r' (\mbox{or } r'')$ and $\bar{c}_4=\bar{c}_4'$
in  {\rm(\ref{f4positive})} and {\rm(\ref{f4negative})},
$\Sigma^4=(f_{4}^{*}(r;H,c_4,\bar{c}_4)\cup f_4^{**}(r;H,c_4,\bar{c}_4'),r,\theta,\phi)$
is a complete \SC hypersurface in the Schwarzschild interior.
\end{itemize}
\end{prop}

\begin{rem}
For given constant mean curvature $H$, if $c_2=c_4$, then $f_2'(r)=f_4'(r)$.
\end{rem}

\begin{prop}
In case {\rm(c)} of Proposition \ref{prop11}, the \SC hypersurface $\Sigma^4$ is $C^\infty$.
\end{prop}

\section{Complete and smooth SS-CMC hypersurfaces} \label{section5}
In this section we will investigate how to join solutions from different regions at $r=2M$
to construct complete hypersurfaces in the Kruskal extension,
and  discuss the smoothness property at each joint point.
First, we discuss \SC hypersurfaces in region {\tt I, I\!I}, and {\tt I'}.
As in section 4, denote $c_H=\min\limits_{r\in(0,2M)}-Hr^3-r^{\frac32}(2M-r)^{\frac12}$,
and notice that $c_H<-8M^3H$ from Figure \ref{KHPNnew4},
then we have the following theorems.

\begin{thm}\label{thm1}
Given constant mean curvature $H$, $c_1<-8M^3H$, and $\bar{c}_1\in\mathbb{R}$,
it determines a \SC hypersurface $\Sigma^1_{H,c_1,\bar{c}_1}$ in region {\tt I}.
This $\Sigma^1_{H,c_1,\bar{c}_{1}}$ connects smoothly with a \SC hypersurface $\Sigma^2_{H,c_1,\bar{c}_2}$
for some $\bar{c}_2$ in region {\tt I\!I}. Moreover,
\begin{itemize}
\item[\rm(a)] when $c_1<c_H$, the corresponding \SC hypersurface $\Sigma^2_{H,c_1,\bar{c}_2}$ in Schwarzschild
is defined on $(0,2M)$, and $\Sigma^1\cup \Sigma^2$
forms a complete and smooth \SC hypersurface in the Kruskal extension with two ends.
\item[\rm(b)] when $c_1=c_H$, the corresponding \SC hypersurface $\Sigma^2_{H,c_1,\bar{c}_2}$ in Schwarzschild
is defined on $(r_H,2M)$, and $\Sigma^1\cup \Sigma^2$
forms a complete and smooth \SC hypersurface in the Kruskal extension with two ends.
Here $r_H$ is defined as in Proposition \ref{prop7}.
\item[\rm(c)] when $c_H<c_1<-8M^3H$, $\Sigma^2_{H,c_1,\bar{c}_2}$
connects smoothly to a \SC hypersurface $\Sigma^3_{H,c_1,\bar{c}_3}$
for some $\bar{c}_3$ in region {\tt I'}.
Then $\Sigma^1\cup\Sigma^2\cup\Sigma^3$
forms a complete and smooth \SC hypersurface in the Kruskal extension with two ends.
\end{itemize}
\end{thm}

\begin{rem} The followings are some descriptions for the ends in each case of Theorem \ref{thm1}.
\begin{itemize}
\item In case (a), among the two ends,
one is toward the space infinity $r=\infty$ in the first Schwarzschild exterior,
and the other is toward the space singularity $r=0$ in the first Schwarzschild interior.
\item In case (b), one end is toward the space infinity in the first Schwarzschild exterior,
and the other end is asymptotic to a cylindrical \SC hypersurface $r=r_H$ in the first Schwarzschild interior.
\item In case (c), the two ends are toward space infinities $r=\infty$ in
the first and second Schwarzschild exteriors, respectively.
\end{itemize}
\end{rem}

Relations between $\bar{c}_1,\bar{c}_2$, and $\bar{c}_3$ are described as below.
\begin{thm}\label{thm1bar}
Consider $c_1<-8M^3H$ and complete \SC hypersurfaces as in Theorem \ref{thm1}.
Take $r_1>2M$ satisfying $r_1+2M\ln|r_1-2M|=0$ and denote $\bar{f}'(r)=f_1'(r)+\frac1{h(r)}$,
which can be expressed as
\begin{align}
\bar{f}'(r)=\frac{r^4}{(Hr^3+c_1)^2+r^3(r-2M)-(Hr^3+c_1)\sqrt{(Hr^3+c_1)^2+r^3(r-2M)}}.\label{barf}
\end{align}
Then
\begin{align*}
\bar{c}_2=\bar{c}_1-\int_{r_2}^{r_1}\bar{f}'(r)\mbox{d}r-(r_2+2M\ln|r_2-2M|),
\end{align*}
where $r_2<2M$ is in the domain of $f_2^{*}(r)$.
In case {\rm(c)} of Theorem \ref{thm1},
if we take $r_2=r_2'=r''$, $\bar{c}_2'=\bar{c}_2$,
and $r_3=r_1$ in {\rm(\ref{f2positive})} and {\rm(\ref{f2negative})},
then
\begin{align*}
\bar{c}_3
&=\bar{c}_2-\int_{r''}^{r_3}\bar{f}'(r)\mbox{d}r-(r''+2M\ln|r''-2M|) \\
&=\bar{c}_1-2\int_{r''}^{r_1}\bar{f}'(r)\mbox{d}r-2(r''+2M\ln|r''-2M|).
\end{align*}
\end{thm}

\begin{proof}[Proof of Theorem \ref{thm1} and \ref{thm1bar}.]
First, we prove that the necessary condition for  $\Sigma^1_{H,c_1,\bar{c}_1}$ and $\Sigma^2_{H,c_2,\bar{c}_2}$
(or $\Sigma^2_{H,c_2,\bar{c}_2}$ and $\Sigma^3_{H,c_3,\bar{c}_3}$) to join smoothly
 is $c_1=c_2$ (or $c_2=c_3$).

Given $\Sigma^1_{H,c_1,\bar{c}_1}=(f_1(r;H,c_1,\bar{c}_1),r,\theta,\phi)$,
by Proposition \ref{asym} we know that
\begin{align*}
f_1'(r)
=\frac{-1}{h}\sqrt{\frac{l_{1}^2}{1+l_{1}^2}}
=-\frac1{h(r)}+\mbox{finite term}\quad\quad \mbox{near } r=2M.
\end{align*}
The limit of the finite term of $f_1'(r)$ at $r=2M$ is
\begin{align*}
\lim_{r\to 2M^+}\left(\frac{-1}{h}\sqrt{\frac{l_{1}^2}{1+l_{1}^2}}+\frac1{h}\right)
=\frac1{2\left(2MH+\frac{c_{1}}{4M^2}\right)^2}.
\end{align*}
If $\Sigma^2_{H,c_2,\bar{c}_2}$ and $\Sigma^1_{H,c_1,\bar{c}_1}$ join smoothly,
the corresponding function $f_2(r)$ satisfies
\begin{align*}
f_2'(r)=\frac{1}{-h}\sqrt{\frac{l_{2}^2}{l_{2}^2-1}}=\frac1{-h(r)}+\mbox{finite term}
\quad \mbox{near } r=2M,
\end{align*}
and $c_2<-8M^3H$ because the interface of region {\tt I} and {\tt I\!I} is $r=2M, t=\infty$.
Calculating the limit of the finite term of $f_2'(r)$ at $r=2M$ gives
\begin{align*}
\lim_{r\to 2M^-}\left(\frac1{-h}\sqrt{\frac{l_{2}^2}{l_{2}^2-1}}+\frac1h\right)
=\frac1{2\left(-2MH-\frac{c_{2}}{4M^2}\right)^2}.
\end{align*}
Hence the necessary condition to join these two hypersurfaces is
\begin{align*}
\frac1{2\left(2MH+\frac{c_{1}}{4M^2}\right)^2}
=\frac1{2\left(-2MH-\frac{c_{2}}{4M^2}\right)^2}
\Rightarrow c_{2}=c_{1}\quad\mbox{or}\quad c_{2}=-c_{1}-16M^3H.
\end{align*}
Because $c_{1}<-8M^3H$ and $c_{2}<-8M^3H$, it follows that $c_{2}=c_{1}$ is the only choice.

Similarly, if $\Sigma^2_{H,c_2,\bar{c}_2}$ and $\Sigma^3_{H,c_3,\bar{c}_3}$ join smoothly,
their corresponding functions $f_2(r)$ and $f_3(r)$ near $r=2M$ satisfy
\begin{align*}
f_2'(r)=\frac{1}{h}\sqrt{\frac{l_{2}^2}{l_{2}^2-1}}=\frac1{h(r)}+\mbox{finite term},\quad
f_3'(r)=\frac{1}{h}\sqrt{\frac{l_{3}^2}{l_{3}^2-1}}=\frac{1}{h(r)}+\mbox{finite term},
\end{align*}
and $c_3<-8M^3H$ because the interface of region {\tt I\!I} and {\tt I'} is $r=2M,t=-\infty$.
Limits of the finite term of $f_2'(r)$ and $f_3'(r)$ at $r=2M$ are
\begin{align*}
\frac{-1}{2\left(-2MH-\frac{c_{2}}{4M^2}\right)^2}\quad\mbox{and}\quad
\frac{-1}{2\left(-2HM-\frac{c_{3}}{4M^2}\right)^2},
\end{align*}
respectively, so the necessary condition to join $\Sigma^2_{H,c_2,\bar{c}_2}$ and $\Sigma^3_{H,c_3,\bar{c}_3}$ is
\begin{align*}
\frac{-1}{2\left(-2MH-\frac{c_{2}}{4M^2}\right)^2}=\frac{-1}{2\left(-2MH-\frac{c_{3}}{4M^2}\right)^2}
\Rightarrow c_{3}=c_{2}\quad\mbox{or}\quad c_{3}=-c_{2}+16M^3H.
\end{align*}
Hence $c_{3}=c_{2}$ is the only possibility because $c_3<-8M^3H$ and $c_2<-8M^3H$.

Next, we find relations of $\bar{c}_1,\bar{c}_2$ and $\bar{c}_3$.
Notice that for $\Sigma_{H,c_1,\bar{c}_1}^1$ and $\Sigma_{H,c_1,\bar{c}_2}^2$,
they have expressions
\begin{align*}
f_1'(r)=-\frac1{h(r)}+\bar{f}'(r)\quad\mbox{and}\quad f_2'(r)=-\frac1{h(r)}+\bar{f}'(r)\mbox{ when } f_2'(r)>0,
\end{align*}
where $\bar{f}'(r)$ is as in (\ref{barf}).
The function $\bar{f}'(r)$ comes from the finite term of $f_1'(r)$ and $f_2'(r)$ near $r=2M$,
and it is clearly well-defined at $r=2M$.
In addition, when $r>r''$ and $r\neq 2M$, $\bar{f}'(r)$ is the sum of two smooth functions.
So $\bar{f}'(r)$ is finite valued for all $r>r''$.

Since we hope $\Sigma_{H,c_1,\bar{c}_1}^1$ and $\Sigma_{H,c_1,\bar{c}_2}^2$ to join smoothly,
they must satisfy the following condition in null coordinates:
\begin{align*}
&\lim_{r\to 2M^+}V(r)=\lim_{r\to 2M^-}V(r) \\
\Rightarrow\; &\exp\left(\frac{1}{4M}\left(\int_{r_1}^{2M}\bar{f}'(r)\mbox{d}r+\bar{c}_1\right)\right) \\
&=\exp\left(\frac{1}{4M}\left(\int_{r_2}^{2M}\bar{f}'(r)\mbox{d}r+\bar{c}_2+r_2+2M\ln|r_2-2M|\right)\right) \\
\Rightarrow\; &\bar{c}_2=\bar{c}_1+\int_{r_1}^{r_2}\bar{f}'(r)\mbox{d}r-(r_2+2M\ln|r_2-2M|).
\end{align*}

From Proposition \ref{prop8}, when we take $r_2=r_2'=r''$
and $\bar{c}_2'=\bar{c}_2$ in (\ref{f2positive}) and (\ref{f2negative}), it follows that
$\Sigma^2_{H,c_1,\bar{c}_2}=(f_2^{*}(r;H,c_1,\bar{c}_2)\cup f_2^{**}(r;H,c_1,\bar{c}_2'),r,\theta,\phi)$
is a complete \SC hypersurfaces in region {\tt I\!I}.

If $\Sigma^2_{H,c_1,\bar{c}_2}$ and $\Sigma^3_{H,c_1,\bar{c}_3}$ join smoothly and $r_3=r_1$,
since their expressions are
\begin{align*}
f_2'(r)=\frac1{h(r)}-\bar{f}'(r)\mbox{ when } f_2'(r)<0\quad\mbox{and}\quad f_3'(r)=\frac1{h(r)}-\bar{f}'(r),
\end{align*}
they must satisfy
\begin{align*}
&\lim_{r\to 2M^-}U(r)=\lim_{r\to 2M^+}U(r) \\
\Rightarrow \; &-\exp\left(-\frac1{4M}\left(-\int_{r''}^{2M}\bar{f}'(r)\mbox{d}r
+\bar{c}_2-r''-2M\ln|r''-2M|\right)\right) \\
&=-\exp\left(-\frac1{4M}\left(-\int_{r_3}^{2M}\bar{f}'(r)\mbox{d}r
+\bar{c}_3\right)\right) \\
\Rightarrow\; &\bar{c}_3
=\bar{c}_2-\int_{r''}^{r_3}\bar{f}'(r)\mbox{d}r-(r''+2M\ln|r''-2M|) \\
\Rightarrow\; &\bar{c}_3
=\bar{c}_1-2\int_{r''}^{r_1}\bar{f}'(r)\mbox{d}r-2(r''+2M\ln|r''-2M|).
\end{align*}

Finally, we investigate the smoothness of these complete \SC hypersurfaces.
When express $\Sigma^1_{H,c_1,\bar{c}_1}$ and $\Sigma^2_{H,c_1,\bar{c_2}}$
in null coordinates near the joint point,
they both have
\begin{align*}
U&=(r-2M)\exp\left(\frac1{4M}\left(2r-\int_{r_1}^r\bar{f}'(r)\mbox{d}r+r_1+2M\ln|r_1-2M|\right)\right) \\
V&=\exp\left(\frac1{4M}\left(\int_{r_1}^r\bar{f}'(r)\mbox{d}r+r_1+2M\ln|r_1-2M|\right)\right).
\end{align*}
Hence $\Sigma_{H,c_1,\bar{c}_1}^1\cup\Sigma_{H,c_1,\bar{c}_2}^2$ is smooth.

The expressions for $\Sigma^2_{H,c_1,\bar{c}_2}$ and $\Sigma^3_{H,c_1,\bar{c}_3}$
in null coordinates near the joint point are both
\begin{align*}
U&=-\exp\left(-\frac1{4M}\left(\int_{r_1}^r\bar{f}'(r)\mbox{d}r+r_1+2M\ln|r_1-2M|\right)\right), \\
V&=(r-2M)\exp\left(\frac1{4M}\left(2r-\int_{r_1}^r\bar{f}'(r)\mbox{d}r+r_1+2M\ln|r_1-2M|\right)\right).
\end{align*}
Hence $\Sigma_{H,c_1,\bar{c}_2}^2\cup\Sigma_{H,c_1,\bar{c}_3}^3$ is smooth.
\end{proof}

For $c_1>-8M^3H$, the \SC hypersurfaces lie in region {\tt I, I\!I'}, and {\tt I'}.
Denote $C_H=\max\limits_{r\in(0,2M)}-Hr^3+r^{\frac32}(2M-r)^{\frac12}$
and note that $C_H>-8M^3H$ from Figure \ref{KHPNnew3}.
By similar argument, we have the following results.

\begin{thm}\label{thm2}
Given  constant mean curvature $H$, $c_1>-8M^3H$, and $\bar{c}_1\in\mathbb{R}$,
it determines a \SC hypersurface $\Sigma^1_{H,c_1,\bar{c}_1}$ in region {\tt I}.
This $\Sigma^1_{H,c_1,\bar{c}_{1}}$ connects smoothly with a \SC hypersurface $\Sigma^4_{H,c_1,\bar{c}_4'}$
for some $\bar{c}_4'$ in region {\tt I\!I'}. Moreover,
\begin{itemize}
\item[\rm(a)] when $c_1>C_H$, the corresponding \SC hypersurface $\Sigma^4_{H,c_1,\bar{c}_4'}$ in Schwarzschild
is defined on $(0,2M)$, and $\Sigma^1\cup \Sigma^4$
forms a complete and smooth \SC hypersurface in the Kruskal extension with two ends.

\item[\rm(b)] when $c_1=C_H$, the corresponding \SC hypersurface $\Sigma^4_{H,c_1,\bar{c}_4'}$ in Schwarzschild
is defined on $(R_H,2M)$, and $\Sigma^1\cup \Sigma^4$
forms a complete and smooth \SC hypersurface in the Kruskal extension with two ends.
Here $R_H$ is defined as in Proposition \ref{proposition11}.

\item[\rm(c)] when $-8M^3H<c_1<C_H$, $\Sigma^4_{H,c_1,\bar{c}_4'}$
connects smoothly to a \SC hypersurface $\Sigma^3_{H,c_1,\bar{c}_3}$
for some $\bar{c}_3$ in {\tt I'}.
Then $\Sigma^1\cup \Sigma^4\cup \Sigma^3$
forms a complete and smooth \SC hypersurface in the Kruskal extension with two ends.
\end{itemize}
\end{thm}

\begin{rem}The followings are some descriptions for the ends in each case of Theorem \ref{thm2}.
\begin{itemize}
\item In case (a), among the two ends, one is toward the space infinity
$r=\infty$ in the first Schwarzschild exterior,
and the other is toward the space singularity $r=0$ in the second Schwarzschild interior.
\item In case (b), one end is toward the space infinity $r=\infty$ in the first Schwarzschild exterior,
and the other end is asymptotic to a cylindrical \SC hypersurface $r=r_H$ in the second Schwarzschild interior.
\item In case (c), the two ends are toward space infinities $r=\infty$ in
the first and second Schwarzschild exteriors, respectively.
\end{itemize}
\end{rem}

\begin{thm}\label{thm2bar}
Consider $c_1>-8M^3H$ and complete \SC hypersurfaces which are described in Theorem \ref{thm2}.
Take $r_1>2M$ satisfying $r_1+2M\ln|r_1-2M|=0$ and denote $\tilde{f}'(r)=f_1'(r)-\frac1{h(r)}$,
which can be expressed as
\begin{align*}
\tilde{f}'(r)
=\frac{-r^4}{(Hr^3+c_1)^2+r^3(r-2M)+(Hr^3+c_1)\sqrt{(Hr^3+c_1)^2+r^3(r-2M)}},
\end{align*}
Then
\begin{align*}
\bar{c}_4'=\bar{c}_1+\int_{r_4'}^{r_1}\tilde{f}'(r)\mbox{d}r+(r_4'+2M\ln|r_4'-2M|),
\end{align*}
where $r_4'<2M$ is in the domain of $f_4(r)$.
In case {\rm(c)} of Theorem \ref{thm2}, if we take $r_4'=r_4=r''$, $\bar{c}_4'=\bar{c}_4$,
and $r_3=r_1$ in {\rm(\ref{f4positive})} and {\rm(\ref{f4negative})}, then
\begin{align*}
\bar{c}_3
&=\bar{c}_4'+\int_{r''}^{r_3}\tilde{f}'(r)\mbox{d}r+(r''+2M\ln|r''-2M|) \\
&=\bar{c}_1+2\int_{r''}^{r_1}\tilde{f}'(r)\mbox{d}r+2(r''+2M\ln|r''-2M|).
\end{align*}
\end{thm}

When $c_1=-8M^3H$, we have
\begin{thm}\label{thm3}
Given constant mean curvature $H$, $c_1=-8M^3H$, and $\bar{c}_1\in\mathbb{R}$,
it determines a \SC hypersurface $\Sigma^1_{H,c_1,\bar{c}_1}$ in region {\tt I}.
This $\Sigma^1_{H,c_1,\bar{c}_{1}}$ connects with a \SC hypersurface $\Sigma^3_{H,c_1,\bar{c}_3}$
for some $\bar{c}_3$ in region {\tt I'} such that $\Sigma^1\cup\Sigma^3$
forms a complete and smooth \SC hypersurface in the Kruskal extension with two ends.
\end{thm}

\begin{rem}
In Theorem \ref{thm3}, the two ends are toward space infinities $r=\infty$ in
the first and the second Schwarzschild exteriors, respectively.
\end{rem}

\begin{proof}
When $c_1=c_3=-8M^3H$, both $\Sigma^1_{H,c_1,\bar{c}_1}$ and $\Sigma^3_{H,c_3,\bar{c}_3}$
pass through the origin in the Kruskal extension which corresponds to $r=2M$ with finite $t$.
Now we  determine the relation between $\bar{c}_1$ and $\bar{c}_3$.
Since
\begin{align*}
&\left\{
\begin{array}{l}
U(r)=\sqrt{r-2M}\exp\left(\frac1{4M}\left(-\int_{r_1}^rf_1'(x)\mbox{d}x-\bar{c}_1+r\right)\right)\\
V(r)=\sqrt{r-2M}\exp\left(\frac1{4M}\left(\int_{r_1}^rf_1'(x)\mbox{d}x+\bar{c}_1+r\right)\right)
\end{array}\right.\quad\mbox{in region {\tt I}},\\
&\left\{
\begin{array}{l}
U(r)=-\sqrt{r-2M}\exp\left(\frac1{4M}\left(-\int_{r_3}^rf_3'(x)\mbox{d}x-\bar{c}_3+r\right)\right)\\
V(r)=-\sqrt{r-2M}\exp\left(\frac1{4M}\left(\int_{r_3}^rf_3'(x)\mbox{d}x+\bar{c}_3+r\right)\right)
\end{array}\right.\quad\mbox{in region {\tt I'}},
\end{align*}
we have
\begin{align*}
&\left.\frac{\mbox{d}V}{\mbox{d}U}\right|_{U=0}
=\left.\frac{\frac{\mbox{\footnotesize d}V}{\mbox{\footnotesize d}r}}
{\frac{\mbox{\footnotesize d}U}{\mbox{\footnotesize d}r}}\right|_{r=2M}
=\exp\left(\frac{1}{2M}\left(\int_{r_1}^{2M}f_1'(r)\mbox{d}r+\bar{c}_1\right)\right)\quad\mbox{in region {\tt I}}, \\
&\left.\frac{\mbox{d}V}{\mbox{d}U}\right|_{U=0}
=\exp\left(\frac{1}{2M}\left(\int_{r_3}^{2M}f_3'(r)\mbox{d}r+\bar{c}_3\right)\right)\quad\mbox{in region {\tt I'}}.
\end{align*}
The integrals $\int_{r_1}^{2M}f_1'(r)\mbox{d}r$ and $\int_{r_3}^{2M}f_3'(r)\mbox{d}r$
are finite because both $f_1'(r)$ and $f_3'(r)$ are of order $O((r-2M)^{-\frac12})$ when $H\neq 0$
and are $0$ when $H=0$.

The condition to form a $C^1$ complete \SC hypersurface is
$$\bar{c}_3=\bar{c}_1+\int_{r_1}^{2M}f_1'(r)\mbox{d}r-\int_{r_3}^{2M}f_3'(r)\mbox{d}r.$$
We can take $r_1=r_3=2M$, and it gives $\bar{c}_1=\bar{c}_3$.

Now we show smoothness of the \SC hypersurface.
When $H=0$, $\Sigma^1\cup\Sigma^3$ is a straight line through the origin point in the Kruskal extension,
so it is smooth.
When $H\neq 0$, we have $f_1'(r)=h^{-\frac12}(r)F_1(r)$ and $f_3'(r)=h^{-\frac12}(r)F_3(r)$,
where $F_1(r)$ and $F_3(r)$ are smooth functions on $r\geq 2M$ and $F_1(2M)=-F_3(2M)$.
Furthermore, by taking $r_1=r_3=2M$ and $\bar{c}_1=\bar{c}_3$, we have
\begin{align*}
\frac{\mbox{d}V}{\mbox{d}U}=
\left\{\begin{array}{ll}
\exp\left(\frac1{2M}\left(\int_{2M}^rf_1'(x)\mbox{d}x+\bar{c}_1\right)\right)
\left(1+\frac{2h^{\frac12}(r)F_1(r)}{1-h^{\frac12}(r)F_1(r)}\right)&\mbox{in region {\tt I}} \\
\exp\left(\frac1{2M}\left(\int_{2M}^rf_3'(x)\mbox{d}x+\bar{c}_1\right)\right)
\left(1+\frac{2h^{\frac12}(r)F_3(r)}{1-h^{\frac12}(r)F_3(r)}\right)&\mbox{in region {\tt I'}}.
\end{array}\right.
\end{align*}
Because
\begin{align*}
\frac{\mbox{d}^2V}{\mbox{d}U^2}
=\frac{\mbox{d}r}{\mbox{d}U}\left(\frac{\mbox{d}}{\mbox{d}r}\frac{\mbox{d}V}{\mbox{d}U}\right),
\end{align*}
it gives
\begin{align*}
\lim_{r\to 2M}\frac{\mbox{d}^2V}{\mbox{d}U^2}
=\left\{
\begin{array}{ll}
2\sqrt{2M}\frac{F_1(2M)}{M}\exp{\left(\frac1{4M}\left(3\bar{c}_1-2M\right)\right)}
& \mbox{in region {\tt I}} \\
-2\sqrt{2M}\frac{F_3(2M)}{M}\exp{\left(\frac1{4M}\left(3\bar{c}_1-2M\right)\right)}
& \mbox{in region {\tt I'}}.
\end{array}\right.
\end{align*}
Since $F_1(2M)=-F_3(2M)$, we get the \SC hypersurface is $C^2$.

When we rewrite the \SC hypersurface in the coordinates $(T=F(X),X,\theta,\phi)$, the \SC equation becomes
\begin{align*}
F''(X)&+\mbox{e}^{-\frac{r}{2M}}\left(\frac{6M}{r^2}-\frac1r\right)(-F(X)+F'(X)X)(1-(F'(X))^2)\\
&+\frac{12HM\mbox{e}^{-\frac{r}{4M}}}{\sqrt{r}}(1-(F'(X))^2)^{\frac32}=0,
\end{align*}
where the spacelike condition is $1-(F'(X))^2>0$ and $r$ is considered as a function of
$X$ and $T=F(X)$ by (\ref{trans}).
Once we know that the \SC hypersurface is $C^2$, the standard PDE theory (see \cite[Theorem 6.17.]{GT} for example)
implies that the \SC hypersurface is $C^\infty$.
\end{proof}

Figures \ref{CMCPnew} and \ref{CMCNnew} show some complete
\SC hypersurfaces in the Kruskal extension for $H>0$ and $H<0$, respectively.

\begin{figure}
\centering
\psfrag{A}{$r=0$}
\psfrag{B}{$r=r_H$}
\psfrag{K}{$r=R_H$}
\psfrag{C}{$r=2M$}
\psfrag{D}{$c<c_H<-8M^3H$}
\psfrag{E}{$c=c_H<-8M^3H$}
\psfrag{I}{$c_H<c<-8M^3H$}
\psfrag{F}{$c=-8M^3H$}
\psfrag{G}{$-8M^3H<c<C_H$}
\psfrag{J}{$-8M^3H<c=C_H$}
\psfrag{H}{$-8M^3H<C_H<c$}
\includegraphics[height=87mm, width=70mm]{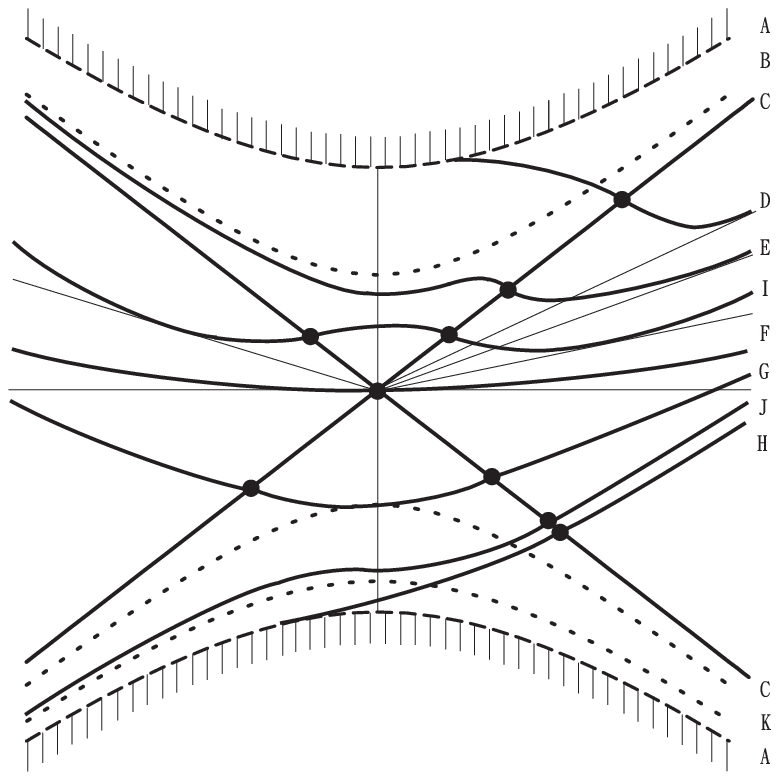}
\caption{\SC hypersufaces with $H>0$.} \label{CMCPnew}
\includegraphics[height=101mm, width=80mm]{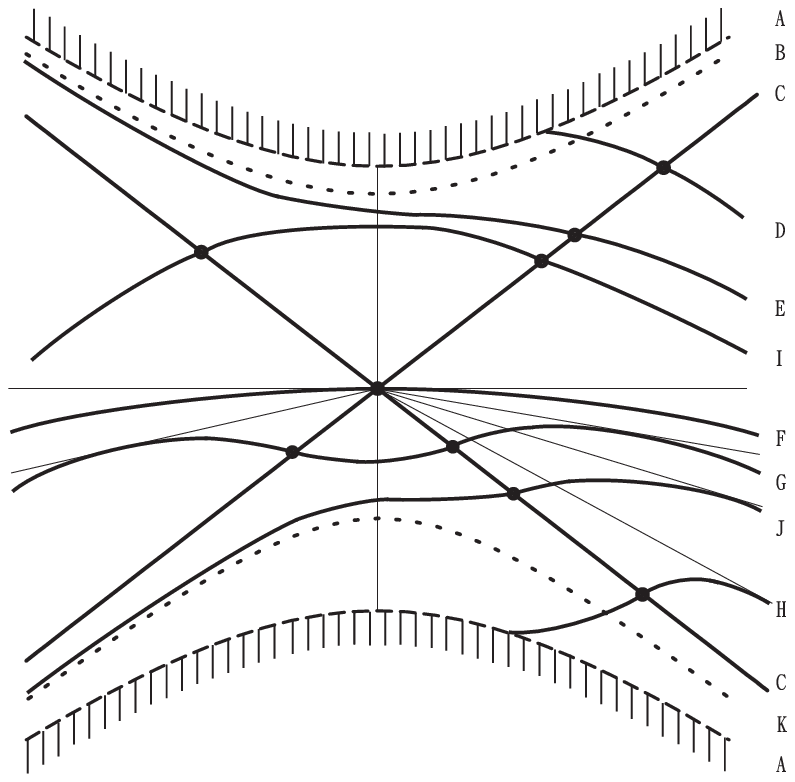}
\caption{\SC hypersufaces with $H<0$.} \label{CMCNnew}
\end{figure}

\begin{rem}\label{rem10}
Besides complete \SC hypersurfaces in Theorems \ref{thm1} and \ref{thm2},
there are different \SC hypersurfaces in the Kruskal extension as below:
\begin{itemize}
\item In case (b) of Proposition \ref{prop8} and $f_2(r)$ is defined on $(0,r_H)$,
the \SC hypersurface is mapped to the region {\tt I\!I} with two ends.
One of them is toward the space singularity $r=0$,
and the other is toward the cylindrical hypersurface $r=r_H$.

\item In case (c) of Proposition \ref{prop8} and $f_2(r)$ is defined on $(0,r']$,
the \SC hypersurface is mapped to the region {\tt I\!I} with two ends.
These two ends are toward the space singularity $r=0$.
This \SC hypersurface is $C^\infty$.

\item In case (b) of Proposition \ref{prop11} and $f_4(r)$ is defined on $(0,R_H)$,
the \SC hypersurface is mapped to the region {\tt I\!I'} with two ends.
One of them is toward the space singularity $r=0$,
and the other is toward the cylindrical hypersurface $r=R_H$.

\item In case (c) of Proposition \ref{prop11} and $f_4(r)$ is defined on $(0,r']$,
the \SC hypersurface is mapped to the region {\tt I\!I'} with two ends.
These two ends are toward the space singularity $r=0$.
This \SC hypersurface is $C^\infty$.

\item We can also start with a \SC hypersurface $\Sigma_{H,c_3,\bar{c}_3}^3$ in region {\tt I'},
and apply similar arguments as Theorems \ref{thm1} and \ref{thm2}.
New complete and smooth \SC hypersurfaces can be found.
They correspond to the case (a), (b) of Theorem \ref{thm1} and \ref{thm2},
which all have one end toward space infinity $r=\infty$ in the second Schwarzschild exterior.
Their another end can be toward space singularity $r=0$ (case (a)) or cylindrical hypersurface (case(b))
in either the first or the second Schwarzschild interior.
\end{itemize}
\end{rem}

\begin{thm}\label{thmall}
All complete \SC hypersurfaces in the Kruskal extension are as in Theorems \ref{thm1}, \ref{thm2}, \ref{thm3},
Remark \ref{rem10}, and cylindrical hypersurfaces.
\end{thm}

\vspace*{5mm}
\fontsize{10}{8pt plus.5pt minus.4pt}\selectfont
\begin{itemize}
\item[] Kuo-Wei Lee
\item[] {\sc Institute of Mathematics, Academia Sinica, Taipei, Taiwan}
\item[] {\it E-mail address}: {\tt d93221007@gmail.com}
\item[]
\item[] Yng-Ing Lee
\item[] {\sc Department of Mathematics, National Taiwan University, Taipei, Taiwan}
\item[] {\sc National Center for Theoretical Sciences, Taipei Office, National Taiwan University, Taipei, Taiwan}
\item[] {\it E-mail address}: {\tt yilee@math.ntu.edu.tw}
\end{itemize}
\end{document}